\documentclass[12pt, reqno]{amsart}
\usepackage[margin=1in]{geometry}
\usepackage{amssymb, amsmath, amsthm,stmaryrd, mathtools}
\usepackage[all]{xy}
\usepackage{amsfonts}
\usepackage{longtable}

\usepackage{tikz}
\usetikzlibrary{positioning}

\usepackage{textcomp,wasysym}
\usepackage{url, verbatim}
\usepackage{graphicx}
\usepackage{caption}

\usepackage{color}

\newtheorem{theorem}{Theorem}[section]
\newtheorem{lemma}[theorem]{Lemma}
\newtheorem{proposition}[theorem]{Proposition}
\newtheorem{corollary}[theorem]{Corollary}

\newtheorem{predefinition}[theorem]{Definition}
\newenvironment{definition}{\begin{predefinition}\rm}{\end{predefinition}}
\newtheorem{preremark}[theorem]{Remark}
\newenvironment{remark}{\begin{preremark}\rm}{\end{preremark}}
\newtheorem{prenotation}[theorem]{Notation}

\newtheorem{example}[theorem]{Example}
\newtheorem{prequestion}[theorem]{Question}

 \makeatletter

\newcommand{\QQ}{\mathbb{Q}}

\newcommand{\NN}{\mathbb{N}}

\newcommand{\ZZ}{\mathbb{Z}}
\newcommand{\FF}{\mathbb{F}}

\newcommand{\NWN}{normalized Weil numbers }
\newcommand{\NWNs}{{\rm NWN}}

\title{Fully maximal and fully minimal abelian varieties}
\date{}
\author{Valentijn Karemaker}
\address{Valentijn Karemaker,  Department of Mathematics,
University of Pennsylvania,
Philadelphia, PA 19104, USA}
\email{vkarem@math.upenn.edu}

\author{Rachel Pries}
\address{Rachel Pries, Department of Mathematics, 
Colorado State University, 
Fort Collins, CO 80523, USA}
\email{pries@math.colostate.edu}

\thanks{Karemaker was partially supported by The Netherlands Organisation for Scientific Research (NWO) through the ``Geometry and Quantum Theory'' research cluster.
Pries was partially supported by NSF grant DMS-15-02227. 
The authors thank Jeff Achter, Gunther Cornelissen, 
Frans Oort, Christophe Ritzenthaler, Jeroen Sijsling, Andrew Sutherland, and a referee for helpful comments.}

\begin{document}

\begin{abstract}
We introduce and study a new way to catagorize supersingular abelian varieties defined over a finite field
by classifying them as {\it fully maximal}, {\it mixed} or {\it fully minimal}.
The type of $A$ depends on the normalized Weil numbers of $A$ and its twists.
We analyze these types for supersingular abelian varieties and curves under conditions on the automorphism group.  
In particular, we present a complete analysis of these properties for supersingular elliptic curves and
supersingular abelian surfaces in arbitrary characteristic,
and for a one-dimensional family  
of supersingular curves of genus $3$ in characteristic $2$.

AMS 2010 mathematics subject classification: primary: 11G10, 11G20, 11M38, 14H37, 14H45; secondary: 11G25, 14G15, 14H40, 14K10, 14K15.



\end{abstract}

\maketitle

\section{Introduction}

Suppose that $X$ is a smooth projective connected curve of genus $g$ defined over a finite field $\FF_q$ of characteristic $p$; write $q=p^r$.
The curve $X$ is {\it supersingular} if the only slope of the Newton polygon of its $L$-polynomial 
is $\frac{1}{2}$ or, 
equivalently, if its normalized Weil numbers are all roots of unity.
If $p=2$, there exists a supersingular curve over $\FF_2$ of every genus \cite{VdGVdV}.
If $p$ is odd, it is not known whether there exists a supersingular curve over $\overline{\FF}_p$ of every genus.
One says that $X$ is \emph{minimal} (resp.\ \emph{maximal}) over $\FF_{q^m}$ if the number 
of $\FF_{q^m}$-points of $X$ realizes the lower (resp.\ upper) bound in the Hasse-Weil theorem.

More generally, 
suppose that $A$ is a principally polarized abelian variety of dimension $g$ defined over $\FF_q$.
Then $A$ is {\it supersingular} if the only slope of its $p$-divisible group $A[p^\infty]$ is $\frac{1}{2}$ or, equivalently, if its normalized Weil numbers are all roots of unity.
One says that $A$ is \emph{minimal} (resp.\ \emph{maximal}) over $\FF_{q^m}$
if Frobenius acts on its $\ell$-adic Tate module by multiplication by 
$\sqrt{q^m}$ (resp.\ -$\sqrt{q^m}$). 
In fact, $A$ (resp.\ $X$) is supersingular if and only if it is minimal over some finite extension of $\FF_q$.

Because of applications to cryptosystems and error-correcting codes, 
there are many papers in the literature about maximal curves but relatively few papers about minimal curves.  
This led to the motivating question: is a supersingular curve
$X/\FF_q$ more likely to be maximal or minimal?
However, this question is not well-posed, since $X$ may be neither until after a finite field extension.
To resolve this, one says that $X/\FF_q$ has parity $1$ if it is maximal after a finite extension of $\FF_q$, and parity $-1$ otherwise, cf.\
Definition \ref{periodinWeilnums}.
The proportion of supersingular elliptic curves with parity $1$ can be determined
using \cite{schoof} (Remark~\ref{Ranalytic}),
but the analogous question for curves of higher genus and abelian varieties of higher dimension is more difficult to answer, since the sizes of the isogeny classes are not known.

In this paper, we address a related question about supersingular curves and abelian varieties, 
based on the fact that most of the supersingular curves found in the literature 
have non-trivial automorphism groups and twists.
The twists of $X/\FF_q$ may have different arithmetic properties.
Specifically, it is possible that $X/\FF_q$ is not maximal over any extension of $\FF_{q}$ but that it has a twist which is
maximal over some extension of $\FF_q$.  From a geometric perspective, 
there is no reason to prefer one twist over another.

The following definition addresses this subtlety.
Suppose that $X/\FF_q$ is a supersingular curve or abelian variety.
We define $X$ to be (i) \emph{fully maximal}, (ii) \emph{fully minimal}, (iii) \emph{mixed} over $\FF_{q}$ 
if (i) all, (ii) none, or (iii) some (but not all) of its $\FF_q$-twists have the property that they are maximal over 
some finite extension of $\FF_q$ 
(Definitions~\ref{type}, \ref{Jactype}).
The type of $X$ depends on its geometric automorphism group, its field of definition, and the \NWN of its twists,
leading to a fascinating interaction between algebra, geometry, and arithmetic.

It is a natural question to ask: under what conditions is a supersingular curve or abelian variety
fully maximal, fully minimal, or mixed over $\FF_q$?
We answer this question for dimension $g=1$ in Section \ref{sec:low}, 
proving that a supersingular elliptic curve is fully maximal over $\FF_p$ if its $j$-invariant is in $\FF_p$ 
and is mixed over $\FF_{p^2}$ otherwise (Theorem \ref{Tconcg1}).
When $g=2$ and $p$ is odd, in Section \ref{Ssurface}, we
give a complete analysis of the three types for simple supersingular abelian surfaces $A$;
in particular, for $A/\FF_{p^r}$
with ${\rm Aut}_{\bar{\FF}_p}(A) \simeq \ZZ/2\ZZ$, then
$A$ is not mixed over $\FF_{p^r}$ if $r$ is odd and $A$ is not fully minimal over $\FF_{p^r}$ if $r$ is even (Proposition \ref{Pmain7}).

The results in Sections \ref{sec:low}-\ref{Ssurface} depend on theoretical results in 
earlier sections which hold for all $g$ and $p$.
Section \ref{minmax} introduces supersingular abelian varieties and curves.
Section \ref{S3} contains information about twists, 
including the bijection between twists of 
$A/\FF_q$ and $\FF_q$-Frobenius conjugacy classes of 
${\rm Aut}_{\bar{\FF}_p}(A)$ (Proposition \ref{bij1})
and the effect of 
twists on the relative Frobenius endomorphism (Proposition \ref{frobtwist}).

In Section \ref{S4}, we study supersingular abelian varieties of arbitrary dimension $g$. 
We characterize the fully maximal, fully minimal, and mixed types 
in terms of arithmetic properties of the normalized Weil numbers of $A/\FF_q$. 
These are roots of unity; the key ingredient for the analysis is the $2$-divisibility of their orders, 
encoded in a multiset $\underline{e}(A/\FF_q)$ (Definition~\ref{Des}).
As an application,
we show that $A$ is not fully minimal over $\FF_{p^r}$ if $A$ is simple and $r$ is even (Proposition \ref{Prevennotfmin}). 
We give a complete characterization of the three types
under the hypothesis that $\vert{\rm Aut}_{\bar{\FF}_p}(A)\vert=2$ (Corollary \ref{typetoei2}), and a criterion for the mixed case
in terms of the orders of the twists and $\underline{e}(A/\FF_q)$ (Corollary \ref{Ctwistodd}).

In Section \ref{S5}, we define the three types
for a supersingular curve $X$.
If $s \equiv 0 \bmod 4$ and
$p \equiv -1 \bmod s$, we prove that
the smooth plane curve $X/\FF_p$ 
with equation $x^s+y^s+z^s=0$ is supersingular and of mixed type over $\FF_p$ (Proposition \ref{Pexamplemixed}).
In Section \ref{SsmallautX}, we study which automorphisms yield parity-changing twists.

Most of the supersingular curves found in the literature are constructed using Artin-Schreier 
theory.  In many cases, the automorphism groups and \NWN of these Artin-Schreier curves
are known, e.g., in \cite{VdGVdV92} and \cite{Bwin}.
An open problem is to determine when these curves are fully maximal, fully minimal, or mixed.
As a result in this direction, we end the paper in Section \ref{SVR}
by studying a one-dimensional
family of supersingular curves $X$ of genus $3$ in characteristic $2$, which 
are $(\ZZ/2\ZZ \times \ZZ/2\ZZ)$-Galois covers of the projective line.
For $X/\FF_{2^r}$, we prove that
$X$ is fully minimal if $r \equiv 0 \bmod 4$, 
$X$ is fully minimal or mixed (with about equal probability) if $r \equiv 2 \bmod 4$,
and $X$ is fully maximal or mixed (with about equal probability) if $r$ is odd 
(Theorem \ref{mainthm}).

\section{Background: supersingular abelian varieties and Weil numbers}\label{minmax}

Let $k = \overline{\FF}_p$.  
Let $A$ be an abelian variety of dimension $g$, 
a priori defined over $k$. 
Throughout the paper, we assume $A$ is defined over a finite field $K = \FF_q$ of cardinality $q = p^r$. 

We write $K$ instead of $\mathrm{Spec}(K)$ when this causes no ambiguity.

\subsection{Frobenius and its characteristic polynomial}\mbox{}

\begin{definition}\label{frobs} \cite[21.2]{oortNP}
Consider the generator $Fr_K: \alpha \to \alpha^q$ of the absolute Galois group 
$G_K = \mathrm{Gal}(k/K)$ of $K$.
If $R$ is a $K$-algebra and $U = \mathrm{Spec}(R)$, then the map which sends $x \mapsto x^q$ for $x \in R$ induces a Frobenius map $f_U$ on $U$. 
The \emph{absolute Frobenius endomorphism} 
$f_A \colon A \to A$
of $A/K$ is the glueing of 
$f_U$ over all open affine subschemes $U$ of $A$. 

For a morphism of $K$-schemes $A \to S$, let $A^{(p)}$ be the fiber 
product of $A \to S \stackrel{f_S}{\leftarrow} S$.
The morphism $f_A$ factors through $A^{(p)}$;
this defines a morphism $\pi = \pi_A \colon A  \to A^{(p)}$ 
called the \emph{relative Frobenius endomorphism}. 
Then \begin{equation}\label{eq:frobs}
\pi_A = f_A \otimes Fr_K^{-1}.
\end{equation}
\end{definition}

By \cite[page 135-138]{tate}, for any $\ell \neq p$, there is a bijection
\begin{equation}\label{tate}
\mathrm{End}_{K}(A) \otimes \QQ_\ell\to \mathrm{End}_{G_{K}}(T_{\ell}(A) \otimes_{\ZZ_{\ell}} \QQ_{\ell}),
\end{equation}
where $T_{\ell}(A)$ denotes the $\ell$-adic Tate module of $A$.
Via this bijection,
$\pi_A$ can be viewed as a linear operator on $T_{\ell}(A) \otimes_{\ZZ_{\ell}} \QQ_{\ell}$. 
Since $\pi_A$ is semisimple (cf. \cite[page 138]{tate}), 
this linear operator is diagonalizable over $\overline{\QQ}_{\ell}$.
Moreover, the characteristic polynomial 
$P(A/K,T)$ of $\pi_A$ (in the sense of \cite[page 110]{Lang}) coincides with that of its 
corresponding linear operator, by e.g., \cite[Chapter VII, Theorem 3]{Lang}. 

\subsection{Weil numbers and zeta functions}\mbox{}\\

The characteristic polynomial $P(A/\FF_{q},T)$ of $\pi_A$ is a monic polynomial in $\ZZ[T]$ of degree $2g$.  
Writing $P(A/\FF_{q},T) = \prod_{i=1}^{2g}(T-\alpha_i)$, the roots $\alpha_i \in \overline{\QQ}$ all satisfy
$\vert \alpha_i \vert = \sqrt{q}$. 

\begin{definition}\label{NWNAV}
The roots $\{\alpha_1, \ldots, \alpha_{2g} \} = \{\alpha_1, \bar{\alpha}_1, \ldots, \alpha_{g}, \bar{\alpha}_g\}$ of $P(A/\FF_{q},T)$ are the \emph{Weil numbers} of $A$.
The \emph{normalized Weil numbers} of $A/\FF_q$ are $\NWNs(A/\FF_q)=\{z_1, \bar{z}_1, \ldots, z_{g}, \bar{z}_g\}$, where $z_i = \frac{\alpha_i}{\sqrt{q}}$.
\end{definition}

In writing the normalized Weil numbers, we use the convention that $\zeta_n = e^{2\pi i/n}$.

\begin{theorem}\label{zetaAV} \cite[Chapter II, Section 1]{milneAV}, \cite[Theorem 1.6]{deligne}, \cite[\S IX, 71]{weil2}
The zeta function of $A$ over $\FF_{q}$ satisfies
\[
Z(A/\FF_{q},T) := \exp\left(\sum_{m\geq 1}\vert A(\FF_{q^m})\vert\frac{T^m}{m}\right) = \frac{P_1(T)\cdot\ldots\cdot P_{2g-1}(T)}{P_0(T)P_2(T)\cdot\ldots\cdot P_{2g-2}(T)P_{2g}(T)},
\]
where
$P_s(T) \in \ZZ[T]$ and $P_s(T) = \prod_{\sigma \in S_s}(1-\alpha_{\sigma}T)$
where $S_s$ is the set of subsets $\sigma=\{i_1, \ldots, i_s\}$ of $\{1, \ldots, 2g\}$ of cardinality $s$ 
and $\alpha_\sigma =  \alpha_{i_1}\alpha_{i_2}\cdot\ldots\cdot \alpha_{i_s}$.
\end{theorem}

Note that $P(A/\FF_{q},T) = T^{2g}P_1(T^{-1})$. The polynomials $P_i(T)$ describe the action of Frobenius on the $i$-th {\'e}tale cohomology of $A/\FF_q$. By \cite[Theorem 1]{tate}, two abelian varieties $A_1$ and $A_2$ over $\FF_q$ have the same zeta function
if and only if $P(A_1/\FF_q, T)=P(A_2/\FF_q, T)$, which holds if and only if $A_1$ and $A_2$ are isogenous over $\FF_{q}$.

\begin{corollary}\label{pointsAV} \cite[Chapter II, Theorem 1.1]{milneAV}
The number of $\FF_q$-points of $A$ satisfies
\[
\vert A(\FF_{q})\vert = \mathrm{deg}(\pi_{A/\FF_{q}} - \mathrm{id}) = P(A/\FF_{q},1) = \prod_{i=1}^{2g} (1-\alpha_i); \ {\rm and \ thus}
\]
\[
\vert \vert A(\FF_{q})\vert - q^{g} \vert \leq 2gq^{(g-\frac{1}{2})}+(2^{2g}-2g-1)q^{(g-1)}.
\]
\end{corollary}

\subsection{Zeta functions of curves} \mbox{}\\

Let $X$ be a smooth projective connected curve of genus $g$ defined over $\FF_q$. 

\begin{theorem}\cite[\S IV, 22]{weil},\cite[\S IX, 69]{weil2}\label{zetacurve} 
The zeta function of $X/\FF_q$ can be written as
\[
Z(X/\FF_q,T) = \frac{L(X/\FF_{q},T)}{(1-T)(1-qT)}
\]
where the \emph{$L$-polynomial} $L(X/\FF_{q},T) \in \ZZ[T]$ of $X/\FF_q$ has degree $2g$ and 
factors as
\[
L(X/\FF_{q},T) = \prod_{i=1}^{2g}(1-\alpha_iT).
\]
\end{theorem}

Then $P(\mathrm{Jac}(X)/\FF_{q},T) = T^{2g}L(X/\FF_{q},T^{-1})$ is the characteristic polynomial of $\pi_{\rm{Jac}(X)}$.
The (normalized) \emph{Weil numbers} of $X$ are 
the (normalized) roots of $P(\mathrm{Jac}(X)/\FF_{q},T)$. 

\begin{corollary}\label{pointscurve}
Let $\{\alpha_1, \bar{\alpha}_1, \ldots, \alpha_{g}, \bar{\alpha}_{g}\}$ be the Weil numbers of $X$.
The number of $\FF_q$-points of $X$ satisfies
$\vert X(\FF_{q}) \vert = q+1-\sum_{i=1}^{g}(\alpha_i + \bar{\alpha}_i)$, 
which implies the Hasse-Weil bound:
\[
\vert \vert X(\FF_{q}) \vert - (q+1)\vert \leq 2g\sqrt{q}.
\]
\end{corollary}

\subsection{Supersingular abelian varieties and curves}\mbox{}

\begin{definition}\label{ssAV}
An abelian variety $A$ is \emph{supersingular} if the only slope of the $p$-divisible group $A[p^\infty]$ is $\frac{1}{2}$.
A curve $X$ is \emph{supersingular} if its Jacobian ${\rm Jac}(X)$ is supersingular. 
\end{definition}

\begin{theorem}\label{propssAV} 
Suppose that $A/\FF_q$ is an abelian variety of dimension $g$.
The following properties are each equivalent to $A$ being supersingular:
\begin{enumerate}
\item the ($q$-normalized) Newton polygon of $P(A/\FF_{q},T)$ 
is a line segment of slope $\frac{1}{2}$;
\item $A$ is geometrically isogenous to a product of supersingular 
elliptic curves, i.e., 
\newline $A \times_{\FF_{q}} k \sim E^g \times_{\FF_{q}} k$ for an elliptic curve $E$ such that $E[p](k) = \{0\}$, \cite[Theorem~4.2]{oortsub};
\item the formal group of $A$ is geometrically isogenous to $(G_{1,1})^g$, \cite[Section 1.4]{lioort};
\item the normalized Weil numbers of $A/\FF_q$ are roots of unity, \cite[Theorem 4.1]{manin2}.
\end{enumerate}
\end{theorem} 

\subsection{Maximal and minimal}

\begin{definition}\label{minmaxAV}
An abelian variety $A/\FF_q$ or a curve $X/\FF_q$
is \emph{maximal} (resp.\ \emph{minimal}) if its normalized Weil numbers 
all equal $-1$ (resp.\ $1$).
\end{definition}

By Corollaries \ref{pointsAV} or \ref{pointscurve}, $|A(\FF_q)|$ or $|X(\FF_q)|$ realizes its upper (resp.\ lower) bound
exactly when $A$ or $X$ is maximal (resp.\ minimal).
A necessary condition for maximality or minimality is that $q$ is a square
(i.e., $r$ is even), by Theorem \ref{zetaAV} or \ref{zetacurve}. 
Also $X/\FF_q$ is maximal (resp.\ minimal) if and only if 
$L(X/\mathbb{F}_q, T) = (1 + \sqrt{q}T)^{2g}$ (resp.\ $(1-\sqrt{q}T)^{2g}$).

The following facts are well-known and hold for curves as well as for abelian varieties, cf. \cite[Theorem 1.9]{vianarodriguez} and \cite[Theorem V.1.15(f)]{stichtenothII}.

\begin{lemma}\label{baseAV}
\begin{enumerate}
\item If $P(A/\mathbb{F}_{q},T) = \prod_{i=1}^{2g} (T-\alpha_i)$, then $P(A/\mathbb{F}_{q^m},T)  = \prod_{i=1}^{2g} (T-\alpha_i^m)$.

\item If $A/\FF_q$ is minimal or maximal, then it is supersingular. 
Conversely, if $A/\FF_{q}$ is supersingular, then it is minimal over some finite extension of $\FF_{q}$.

\item

\begin{enumerate}
\item If $A/\mathbb{F}_q$ is maximal, then $A/\mathbb{F}_{q^m}$ is maximal for odd $m$ and minimal for even $m$.
\item If $A/\mathbb{F}_q$ is minimal, then $A/\mathbb{F}_{q^m}$ is minimal for all $m \in {\mathbb N}$. 
\end{enumerate}
\end{enumerate}
\end{lemma}

\section{Twists} \label{S3}

Let $K = \FF_{q}$ with $q=p^r$ and let $k=\overline{\FF}_p$.
For $m \in \mathbb{N}$, let $K_m$ be the unique extension of $K$ of degree $m$.
Let $Fr_K$ be the generator of $G_K = \mathrm{Gal}(k/K)$ as in Definition \ref{frobs}. 

In this section, we review the theory of twists of abelian varieties 
following \cite{serregal} and \cite{CHdescent}.

\subsection{Twists, cocycles, and Frobenius conjugacy classes}\mbox{}\\

Let $A/K$ be a principally polarized abelian variety of dimension $g$. 
We restrict to automorphisms of $A$ that are compatible with the principal polarization $\lambda$. 
For ease of notation, we write $A$ instead of $(A,\lambda)$ and $\mathrm{Aut}_k(A)$ instead of $\mathrm{Aut}_k(A,\lambda)$.

\begin{definition}\label{twist} A ($K$-)\emph{twist} of $A/K$ is an abelian variety $A'/K$ for which there exists a geometric isomorphism 
\begin{equation}\label{twistiso}
\phi: \bar{A} \xrightarrow{\simeq} \bar{A}',
\end{equation}
where $\bar{A}=A \times_K k$ and $\bar{A}'=A' \times_K k$.  A twist $A'/K$ is \emph{trivial} if $A \simeq_{K} A'$.
Let $\Theta(A/K)$ denote the set of $K$-isomorphism classes of twists $A'/K$ of $A/K$.
\end{definition}

\begin{definition}\label{twistphi}
Given $\sigma \in G_K$ and $\phi:  \bar{A} \xrightarrow{\simeq} \bar{A'}$, 
let $\prescript{\sigma}{}{\phi} \colon \bar{A} \xrightarrow{\simeq} \bar{A'}$ 
denote the (twisted) isomorphism which acts on $x \in \bar{A}(k)$ 
via $\prescript{\sigma}{}{\phi}(x) =\sigma(\phi(\sigma^{-1}(x)))$
or, more precisely, via
\[
\prescript{\sigma}{}{\phi} = (\mathrm{id}_{A'} \times_{\mathrm{Spec}(K)} \mathrm{Spec}(\sigma)) \circ \phi \circ (\mathrm{id}_{A} \times_{\mathrm{Spec}(K)}\mathrm{Spec}(\sigma))^{-1}.
\]


Similarly, if $A' = A$ and $\tau \in  \mathrm{Aut}_{k}(A)$, let $\prescript{{Fr_K}}{}{\tau}$ denote the (twisted) automorphism, which acts on $x \in \bar{A}(k)$ by
\[
\prescript{{Fr_K}}{}{\tau}(x) = {Fr_K}(\tau({Fr_K^{-1}}(x))).
\]
\end{definition}

\begin{definition}\label{frobcon} Two automorphisms $g,h \in \mathrm{Aut}_{k}(A)$ are \emph{$K$-Frobenius conjugate} 
if there exists $\tau \in \mathrm{Aut}_{k}(A)$ such that
\[
g = \tau^{-1} h(\prescript{{Fr_K}}{}{\tau}).
\]
In particular, $g$ is $K$-Frobenius conjugate to ${\rm id}$ if $g=\tau^{-1} (\prescript{{Fr_K}}{}{\tau})$ for some $\tau \in \mathrm{Aut}_{k}(A)$.
\end{definition}

\begin{remark}\label{rem:fielddef}
If all automorphisms of $A$ are defined over $K$, then $G_{K}$ acts trivially on $\mathrm{Aut}_k(A)$.
(By \cite[Theorem 2(d)]{tate}, this is true if $A$ is maximal or minimal over $K$.) 
In this case, the $K$-Frobenius conjugacy classes are the same as standard conjugacy classes. 
\end{remark}

\begin{proposition} \label{bij1} 
\cite[Proposition III.5]{serregal}, \cite[Proposition 1]{serreloc}, (see also \cite[Propositions 5,9]{meatop} for curves)
Given $\phi: \bar{A} \stackrel{\simeq}{\to} \bar{A}'$ as in \eqref{twistiso}, 
consider the cocycle $\xi_{\phi} \colon G_K \to \mathrm{Aut}_{k}(A)$ defined by 
\begin{equation}\label{cocyc}
\xi_{\phi}(\sigma) = \phi^{-1} \circ \prescript{\sigma}{}{\phi}.
\end{equation}
Next, for any $\xi \in C^1(G_K,\mathrm{Aut}_{k}(A))$, let 
\begin{equation} \label{Eaut}
g_\xi=\xi(Fr_K) \in \mathrm{Aut}_k(A).
\end{equation}
The maps taking $\phi \mapsto \xi_\phi \mapsto g_\phi := g_{\xi_{\phi}}$ yield bijections:
\begin{equation} \label{Ebijection}
\Theta(A/K) \to H^1(G_K,\mathrm{Aut}_{k}(A)) \to \{\textrm{$K$-Frobenius conjugacy classes of } \mathrm{Aut}_{k}(A)\}.
\end{equation}
\end{proposition}

Given $g \in {\rm Aut}_k(A)$, let $\xi_g \in C^1(G_K,\mathrm{Aut}_{k}(A))$ be the cocycle such that $\xi_g(Fr_K)=g$
and let $\phi_g: \bar{A} \xrightarrow{\simeq} \bar{A}'$ be such that $\xi_{\phi_g} = \xi_g$.
Note that $\phi_g$ is not uniquely determined: 
if $\tau \in {\rm Aut}_k(A)$ is such that $\tau^{-1}g \prescript{{Fr_K}}{}{\tau} = g$, 
then $\phi' = \phi_g \circ \tau: \bar{A} \xrightarrow{\simeq} \bar{A}'$ also has the property that $\xi_{\phi'} = \xi_g$. 
In this case, $\tau$ is defined over $K$, so $\phi \circ \tau$ and $\phi$ have the same field of definition.

\begin{definition}\label{order}
The \emph{order} of a twist $A'/K$ is the smallest $m \in \mathbb{N}$ such that over the degree $m$ extension $K_m$ of $K$ there exists an isomorphism 
$\phi: A \times_K K_m \xrightarrow{\simeq} A' \times_K K_m$.
\end{definition}

If $A'/K$ is a twist of order $m$ and $\phi: \bar{A} \xrightarrow{\simeq} \bar{A}'$
is an isomorphism, then Definition \ref{order} implies that 
$\phi \circ \tau$ is defined over the degree $m$ extension $K_m$ of $K$ for some $\tau \in {\rm Aut}_k(A)$.

\begin{remark}\label{Rcompose}
If $T \in {\mathbb N}$, then 
\begin{equation} \label{Ecompose}
\xi_g(Fr^T_{K}) = g(\prescript{{Fr_K}}{}{g})(\prescript{{Fr^2_K}}{}{g}) \cdots (\prescript{{Fr^{T-1}_K}}{}{g}).
\end{equation}
Given $\phi : \bar{A} \xrightarrow{\simeq} \bar{A}'$, 
write $g:= g_{\phi}$ and let $T_g$ be the smallest $T \in \mathbb{N}$ such that $\xi_g(Fr_K^T) = {\rm id}$.
Then $T_g$ is the degree of the field of definition of $\phi$ over $K$.
\end{remark}

\begin{lemma}  \label{LTcG}
Let $c_g$ be the smallest $c \in \mathbb{N}$ such that $\xi_g(Fr_K^c)$ is defined over $K_c$.
Then $c_g$ divides $T_g$ 
and $T_g/c_g$ equals the order of $G:=g(\prescript{{Fr_K}}{}{g})(\prescript{{Fr^2_K}}{}{g}) \cdots (\prescript{{Fr^{c_g-1}_K}}{}{g})$.
\end{lemma}

\begin{proof}
When $c_g=1$, the result is immediate, since $g$ is defined over $K$ and $G = g$.  

Now suppose that $c_g >1$.
By Remark \ref{Rcompose}, the twist is an element $A'$ of the set $\Theta(A, K_{T_g}/K)$ of twists $A'/K$ of $A/K$ such that $A \times_K K_{T_g} \simeq_{K_{T_g}} A' \times_K K_{T_g}$.
The bijection
$\theta \colon \Theta(A,K_{T_g}/K) \to H^1(\mathrm{Gal}(K_{T_g}/K),\mathrm{Aut}_{K_{T_g}}(A))$
from \cite[Proposition III.5]{serregal} shows that $A'$ corresponds to the automorphism
$\xi_g(Fr_K) = g$ in $\mathrm{Aut}_{K_{T_g}}(A)$. 
It follows that $g$ (and thus $G$) is defined over $K_{T_g}$.
Hence, $K_{c_g} \subset K_{T_g}$ and $c_g \vert T_g$.

The base changes $A_{c_g} = A \times_K K_{c_g}$ and $A'_{c_g}=A' \times_K K_{c_g}$
first become isomorphic over $K_{T_g}K_{c_g} = K_{T_g}$.
So $\phi$ is defined over an extension of $K_{c_g}$ of degree $T'=[K_{T_g}: K_{c_g}] = T_g/c_g$.  
The automorphism corresponding to the twist over $K_{c_g}$ is $G$.
Hence, replacing $g$ by $G$, the conclusion follows from the case when $c_g=1$.
\end{proof}

\subsection{Effect of a twist on the Frobenius endomorphism}\mbox{}\\

In this section, we study how twisting $A/K_c$ by $G \in {\rm Aut}_{K_c}(A)$ affects the relative 
Frobenius endomorphism $\pi=\pi_A \in \mathrm{End}_{K_c}(A)$ of $A$ and 
the normalized Weil numbers of $A$ over $K_c$.

\begin{proposition}\label{frobtwist}
Suppose that $A$ is defined over $K_c$
and $\phi : A \times_{K_c} k \xrightarrow{\simeq} A' \times_{K_c} k$ 
is a geometric isomorphism.
Suppose that $G_{\phi} = \xi_{\phi}(Fr_{K_c})$ is in ${\rm Aut}_{K_c}(A)$.
Then the relative Frobenius endomorphism $\pi'$ of $A'$ satisfies
\begin{equation} \label{eqnfrobtwist}
\phi^{-1} \circ \pi'\circ \phi = \pi_A \circ G_{\phi}^{-1}.
\end{equation}
\end{proposition}

\begin{remark} 
The right hand side of \eqref{eqnfrobtwist} is defined over $K_c$, so the left hand side is as well.
In particular, 
$\pi'$ and $\pi_A \circ G_{\phi}^{-1}$ 
have the same characteristic polynomial.
\end{remark}

\begin{proof}
Let $f' = f_{A'}$ be the absolute Frobenius endomorphism of $A'$.
By \eqref{eq:frobs}, $\pi_A = f_A \otimes Fr_{K_c}^{-1}$ and $\pi' = f_{A'} \otimes Fr_{K_c}^{-1}$.
Also, $f = \phi^{-1} \circ f' \circ \phi$.
Furthermore, by \eqref{cocyc},
\[
G_{\phi}^{-1} = (\mathrm{id}_A \otimes Fr_{K_c}) \circ \phi^{-1} \circ (\mathrm{id}_A \otimes Fr_{K_c}^{-1}) \circ \phi.
\]
Hence, as in \cite[Proposition 11]{meatop}, 
\[
\begin{aligned}
 \phi^{-1} \circ \pi' \circ \phi &= \phi^{-1} \circ \left(f' \otimes Fr_{K_c}^{-1} \right) \circ \phi \\
 &= \phi^{-1} \circ \left(\left(\phi \circ f \circ \phi^{-1} \right) \otimes Fr_{K_c}^{-1} \right) \circ \phi\\
 &= \left(f \otimes Fr_{K_c}^{-1}\right) \circ \left( \mathrm{id}_A \otimes Fr_{K_c}\right) \circ \phi^{-1} \circ \left( \mathrm{id}_A \otimes Fr_{K_c}^{-1}\right) \circ \phi \\
 &= \pi_A \circ G_{\phi}^{-1}.
 \end{aligned}
\]
\end{proof}

\subsection{Twists by automorphisms of order $2$}\mbox{}

\begin{lemma} \label{Lquadtrivial}
Given $\phi:  \bar{A} \xrightarrow{\simeq} \bar{A}'$, 
if $g_\phi \in \mathrm{Aut}_K(A)$ has order $2$, 
then the twist $A'/K$ is either quadratic or trivial.
It is trivial if and only if $g_\phi$ is $K$-Frobenius conjugate to ${\rm id}$.
\end{lemma}

\begin{proof}
Write $g=g_\phi$.  By hypothesis, $c_g = 1$, so by Lemma \ref{LTcG}, 
$T_g = \vert g\vert =2$.
By Definition \ref{order}, the order of the twist is at most $2$.
The last statement follows from Proposition \ref{bij1}.
\end{proof}

The conclusion of Lemma \ref{Lquadtrivial} can be false if $g_\phi$ is not defined over $K$.

\begin{definition}\label{iota} 
Let $\iota \in \mathrm{End}_K(A) \otimes \QQ_\ell$ correspond 
to $[-1] \in \mathrm{End}_{G_K}(T_{\ell}(A) \otimes_{\ZZ_{\ell}} \QQ_{\ell})$ under the bijection in \eqref{tate}.
Then $\iota$ is defined over $K$ and central in $\mathrm{Aut}_k(A)$. 
Let $A_\iota$ denote the $K$-twist of $A$ for $\iota$. 
By Lemma \ref{Lquadtrivial}, $A_{\iota}/K$ is either a trivial or a quadratic twist. 
\end{definition}

By Proposition \ref{frobtwist}, if $A/K$ is maximal, then $A_\iota/K$ is minimal, and vice versa.
Conversely, the next result shows 
that $\iota$ is the only automorphism whose twist can switch between the  maximal and minimal conditions.  
We generalize this result in Corollary~\ref{Ctwistodd}.

\begin{proposition}\label{maxtomin}
Suppose that $\phi: A \times_K k \stackrel{\simeq}{\to} A' \times_K k$ where $A/K$ is maximal and $A'/K$ is minimal (or vice versa). 
Then $g_\phi = \iota$ and $A'/K \simeq A_\iota/K$ is a quadratic twist of $A/K$.

Replacing $A$ by a curve $X$, the same conclusions are true and $X$ is also hyperelliptic.
\end{proposition}

\begin{proof}
By Definition \ref{minmaxAV}, $P(A/K, T)$ and $P(A'/K, T)$ split completely into linear factors over $\QQ$. 
Thus the linear operators corresponding to $\pi_A$ and $\pi_{A'}$ under \eqref{tate} are 
diagonalizable over $\QQ_{\ell}$. 
So $\pi_A = \sqrt{q} \cdot \iota$ and $\pi_{A'} = \sqrt{q} \cdot {\rm id}$ in $\mathrm{End}_K(A) \otimes \QQ_\ell$. 
By Proposition \ref{frobtwist}, this implies that 
$g_{\phi}=\xi_{\phi}(Fr_K)$ is $K$-Frobenius conjugate to $\iota$. 
So $g_{\phi} = \tau^{-1}\iota \prescript{{Fr_K}}{}{\tau}$ for some $\tau \in \mathrm{Aut}_k(A)$. 

Since $A/K$ is maximal, $\mathrm{Aut}_k(A) = \mathrm{Aut}_{K}(A)$ \cite[Theorem 2d]{tate}. 
In particular, $\prescript{{Fr_K}}{}{\tau} = \tau$.
Because $\iota$ is central in $\mathrm{Aut}_k(A)$, the $K$-Frobenius conjugacy class of $\iota$ consists of one element. 
Thus $g_{\phi} = \iota$ and $A'/K \simeq A_\iota/K$.
Moreover, $\iota$ satisfies the conditions of Lemma \ref{Lquadtrivial}. 
Since $A \not\simeq_K A'$, the twist $A'/K$ is nontrivial and thus quadratic.

The same conclusions are true replacing $A$ by a curve $X$.  
Also, $X$ is hyperelliptic because the quotient of $X$ by $\iota$ has genus $0$, since the trivial eigenspace for the action of $\iota$ is trivial.
\end{proof}

\section{Fully maximal, fully minimal, and mixed abelian varieties} \label{S4}

Let $K=\FF_q$ with $q=p^r$ and let $k = \overline{\FF}_p$. 
Let $A$ be a principally polarized supersingular abelian variety of dimension $g$ defined over $K$.
Let $\NWNs(A/K)=\{z_1, \bar{z}_1, \ldots, z_g, \bar{z}_g\}$ be the \NWN of $A/K$, as in Definition \ref{NWNAV}.  

\subsection{Period, parity, and types}\label{Sobser}

\begin{definition}\label{periodinWeilnums}\mbox{}
\begin{enumerate}
\item The {\it $\FF_q$-period} $\mu(A)$ of $A$ 
is the smallest $m \in {\mathbb N}$ such that $q^m$ is square 
and 
\begin{itemize}
\item[(i)] $z_i^m = -1$ for all $1 \leq i \leq g$, or
\item[(ii)] $z_i^m =1$ for all $1 \leq i \leq g$. 
\end{itemize}
\item The {\it $\FF_q$-parity} $\delta(A)$ is $1$ in case (i) and is $-1$ in case (ii).
\end{enumerate}
\end{definition}

In other words, the period is the smallest $m \in \mathbb{N}$ such that $\pi_{A/\FF_{q^m}}$ is $\sqrt{q^m}[\pm 1]$.
The definition of the period and parity is compatible with \cite[page $144$]{stichxing}\label{periodparitycurve}.
Note that $A$ is maximal (resp.\ minimal) over $\FF_q$ if and only if $\mu(A) =1$ and $\delta(A)=1$
(resp.\ $\delta(A)=-1$).
%

Let $\Theta(A/K)$ be the set of $K$-isomorphism classes of twists $A'/K$ of $A$, see Definition \ref{twist}.

\begin{definition}\label{type}
A principally polarized supersingular abelian variety $A/K$ is of one of the following \emph{types} over $K$:
\begin{enumerate}
\item \emph{fully maximal} if $A'/K$ has $K$-parity $\delta=1$ for all $A' \in \Theta(A/K)$;
\item \emph{fully minimal} if $A'/K$ has $K$-parity $\delta=-1$ for all $A' \in \Theta(A/K)$;
\item \emph{mixed} if there exist $A', A'' \in \Theta(A/K)$ with $K$-parities $\delta(A')=1$ 
and $\delta(A'')=-1$.
\end{enumerate}
\end{definition}

If $A/K$ has $K$-period $1$, then $A/K$ is maximal or minimal and so $A$ is mixed over $K$ 
since $A_{\iota}$ has the opposite parity.
For this reason, the terminology is better suited for curves than for abelian varieties, see Lemmas \ref{Lbasicdef} and \ref{Lbasicfact}.  Also, it is most interesting to study the type of $A/K$ over small fields of definition.

\begin{example}
Let $p \equiv 3 \bmod 4$ with $p > 3$.
The supersingular elliptic curve $E:y^2=x^3-x$ has ${\rm Aut}_k(E) \simeq \ZZ/4\ZZ$.
Then $\NWNs(E/\FF_p) = \{\pm i\}$ and $\NWNs(E/\FF_{p^2})= \{-1,-1\}$.  
So $E$ has two $\FF_p$-twists and is fully maximal over $\FF_p$.
It has four $\FF_{p^2}$-twists and is mixed over $\FF_{p^2}$ since an automorphism of order $4$ acts on
$\NWNs(E/\FF_{p^2})$ by multiplication by $\pm i$.
Cf.\ Lemma~\ref{LEextraaut}.
\end{example}

Let $n$ be odd.  The parity is preserved under a degree $n$ extension,
i.e., $\delta(A \times_K K_n)=\delta(A)$. 
Hence, if $A/K$ is mixed, then $A \times_K K_n$ is also mixed: 
if $A'/K$ is a twist with opposite parity from $A/K$, then $A' \times_K K_n$ is a twist 
of opposite parity from $A \times_K K_n$.
Motivated by this, we measure the $2$-divisibility of the orders of the period in the next section.  
 
\subsection{Relationship between types and Weil numbers}\mbox{}\\

By Theorem \ref{propssAV}, 
the \NWN $\{z_1, \ldots, z_g\}$ of a supersingular abelian variety $A/K$ are roots of unity in ${\mathbb C}^*$. 
If $z \in {\mathbb C}^*$ is a root of unity, let $o(z)$ denote its multiplicative order in ${\mathbb C}^*$.
We measure the $2$-divisibility of $o(z_i)$ in the next definition.  
 
\begin{definition} \label{Des}
Let $e_i={\rm ord}_2(o(z_i))$.
The \emph{$2$-valuation vector} of $A/K$ is the multiset 
$\underline{e}=\underline{e}(A/K):= \{e_1, \ldots, e_g\}$.
The notation $\underline{e}=\{e\}$ means that $e_i=e$ for $1 \leq i \leq g$.
\end{definition} 

Write $o(z_i) =2^{e_i} c_i$ with $c_i$ odd. 
Then $z_i^m=-1$ for some $m \in \NN$ if and only if $e_i \geq 1$. Also:
\begin{eqnarray}
{\rm ord}_2(o(z) ) = 1 \Leftrightarrow {\rm ord}_2(o(-z) ) = 0;
\ {\rm if \ } {\rm ord}_2(o(z) ) \geq 2, {\rm \ then \ } {\rm ord}_2(o(-z) ) \geq 2; \label{negate} \\
\text{
If $r$ is odd, then $\underline{e} \not = \{0\}, \{1\}$, because 
$P(A/K,T) \in \ZZ[T]$.} \label{no01rodd}
\end{eqnarray}

\begin{remark}
For the $\FF_q$-parity, note that $\delta(A)=1$ if and only if $\underline{e} = \{e\}$ with $e \geq 1$ (or $e \geq 2$ when $r$ is odd).
For the $\FF_q$-period, write $\mu(A)=2^E\bar{\mu}$ where $\bar{\mu}$ is odd.  If $\underline{e} = \{e\}$, 
then $E={\rm max}(e-1,0)$. If $\underline{e}$ is not constant, then $E={\rm max} \{e_i \mid 1 \leq i \leq g\}$. 
\end{remark}


\begin{lemma} \label{typetoei1}
Let $\underline{e}=\underline{e}(A/K)$.
\begin{enumerate}
\item If $A/K$ is fully maximal, then (i) $\underline{e}=\{e\}$ with $e \geq 2$;
\item If $A/K$ is fully minimal, then (ii) the $e_i$ are not all equal;
\item If (iii) $\underline{e} =\{e\}$ with $e \in \{0,1\}$ 
and $r$ is even, then $A/K$ is mixed. 
\end{enumerate}
\end{lemma}


\begin{proof}
\begin{enumerate}
\item If $A/K$ is fully maximal, then it has $K$-parity $+1$; so 
$\underline{e}=\{e\}$ for some $e \geq 1$ (with $e \geq 2$ if $r$ is odd by \eqref{no01rodd}). 
Suppose that $r$ is even and $\underline{e} = \{1\}$.  By \eqref{negate}, 
the twist $A_{\iota}$ has the property that $\underline{e}=\{0\}$.
So $A_\iota$ has $K$-parity $-1$, which contradicts the fact that $A/K$ is fully maximal.
Thus condition (i) holds.

\item 
If $A/K$ is fully minimal, then it has $K$-parity $-1$. 
By \eqref{no01rodd}, 
either $\underline{e}=\{0\}$ with $r$ even
or the $e_i$ are not all the same. 
If $r$ is even and $\underline{e}=\{0\}$, then the twist by $\iota$ is maximal, giving a contradiction.
Thus condition (ii) holds.

\item This is the contrapositive of parts (1) and (2).
\end{enumerate}
\end{proof}

\begin{proposition} \label{Prevennotfmin}
If $A/\FF_q$ is simple and $q=p^r$ with $r$ even, then $A/\FF_q$ is not fully minimal.
\end{proposition}

\begin{proof}
If $A/\FF_q$ is simple, the Weil numbers $\{\sqrt{q} z_i\}$ are all conjugate over $\QQ$. 
Let $n=o(z_1) $ and $e={\rm ord}_2(n)$.
Since $r$ is even, $\sqrt{q} \in \QQ$, so the conjugates of $\sqrt{q} z_1$ 
are precisely the $\phi(n)$ values $\sqrt{q}\zeta_{n}^j$ for $j \in (\ZZ/n\ZZ)^*$. 
So $\underline{e}=\{e\}$.
By Lemma \ref{typetoei1}(2), $A/\FF_q$ is not fully minimal.
\end{proof}

\subsection{Types of abelian varieties with small automorphism group}

\begin{corollary} \label{typetoei2}
Suppose that $\vert{\rm Aut}_k(A)\vert =2$. 
Then 
\begin{enumerate}
\item $A/K$ is fully maximal if and only if (i) $\underline{e}=\{e\}$ with $e \geq 2$;
\item $A/K$ is fully minimal if and only if (ii) the $e_i$ are not all equal;
\item $A/K$ is mixed if and only if (iii) $\underline{e} =\{e\}$ with $e \in \{0,1\}$ 
and $r$ is even. 
\end{enumerate}
\end{corollary}

\begin{proof}
One set of implications is Lemma \ref{typetoei1}. 
Conversely,
if $\vert{\rm Aut}_k(A)\vert=2$, then $A_K$ has at most one nontrivial twist, which is $A_\iota$.
Thus, $A/K$ is fully maximal (resp.\ fully minimal) if and only if $A$ and $A_\iota$ both have $K$-parity $+1$ (resp.\ $-1$).  
The result follows because negation of $\{z_i\}$ preserves
each of the conditions $(i), (ii), (iii)$ for $\underline{e}$, by \eqref{negate}. 
\end{proof}

By Corollary \ref{typetoei2}, if $\vert{\rm Aut}_k(A)\vert =2$, then 
the type of $A/K$ is preserved under odd degree extensions of $K$.

\begin{remark} \label{Ronemoduli}
Let $S$ be an irreducible component of the supersingular locus of the moduli space of principally polarized abelian varieties 
of dimension $g$.  Among the abelian varieties $A$ represented by $\FF_{q}$-points of $S$, the typical structure of
${\rm Aut}_k(A)$ is not known in general.
For $g \geq 2$ and $p$ odd, one might expect that typically ${\rm Aut}_k(A) \simeq \ZZ/2\ZZ$. 
For $g=2$ and $p$ odd, we prove that this is true in Proposition \ref{Pmaincor}.
\end{remark}

\begin{remark}
Let $S$ and $A$ be as in Remark \ref{Ronemoduli} and $K=\FF_q$.
If $p$ is odd, one expects the proportion of $A$ with $\ZZ/2\ZZ \times \ZZ/2\ZZ \subset {\rm Aut}_K(A)$ to be small.
The reason is that if $\ZZ/2\ZZ \times \ZZ/2\ZZ \subset {\rm Aut}_K(A)$, then $A$ is not simple over $K$ by \cite[Theorem B]{kanirosen}.
So this condition implies that 
the $a$-number of $A$ is at least two, by \cite[Proposition 4]{glasspries}. 
However, for all $g$ and $p$, it is known that $A$ generically has $a$-number $1$ \cite[Section 4.9]{lioort}. 
\end{remark}

\subsection{Parity-changing twists of abelian varieties} \label{SparityAV} \mbox{}\\

Suppose that
$A'/K \in \Theta(A/K)$ is a $K$-twist of $A$ of order $T$.
Then there is an isomorphism 
$\phi: A \times_{K} {K_T} \xrightarrow{\simeq} A' \times_{K} {K_T}$ defined over $K_T$.
Denote $\NWNs(A/K) = \{z_i, \bar{z}_i\}_{1\leq i \leq g}$ and $\NWNs(A'/K)= \{w_i,\bar{w}_i\}_{1 \leq i \leq g}$. 
After possibly reordering, $z_i^T = w_i^T$ and hence
\begin{equation}\label{rela}
w_i = \lambda_i z_i
\end{equation}
for some (not necessarily primitive) $T$-th root of unity $\lambda_i$.
Let $t ={\rm lcm}\{o(\lambda_i)  \mid 1 \leq i \leq g\}$. 

By definition, $t \mid T$.
In particular, if $A'/K$ is a trivial twist, then $t=1$ and $z_i = w_i$ for all $i$.
If $t \not = T$, it means that $A$ and $A'$ are isogenous but not isomorphic over $\FF_{q^t}$. 
Conversely, if \eqref{rela} holds, then $A$ and $A'$ are isogenous, but not necessarily isomorphic, over $K_T$.
 
\begin{lemma} \label{Ltwistsimple}
Let $\underline{e}$ be the $2$-valuation vector of $A/K$.
Suppose that $A'/K$ is a $K$-twist of $A/K$ of order $T$.  Let $\epsilon = {\rm ord}_2(T)$.
If $\epsilon < {\rm min}\{e_i \mid 1 \leq i \leq g\}$, then 
$\underline{e}(A'/K)=\underline{e}$.
\end{lemma} 
 
\begin{proof}
If $w_i = \lambda_i z_i$,
then
$\mathrm{ord}_2(o(w_i) ) \leq \max (\mathrm{ord}_2(o(\lambda_i) ) ,\mathrm{ord}_2(o(z_i) ) )$, with equality if the two values are not equal.
Then ${\rm ord}_2(o(\lambda_i) ) \leq {\rm ord}_2(T) = \epsilon$ so the hypothesis implies that ${\rm ord}_2(o(w_i)) = {\rm ord}_2(o(z_i))$.
\end{proof}
 
\begin{proposition} \label{Ptwistodd}
Suppose that
$A/K$ has $K$-period $M$ and $K$-parity $+1$
and its $K$-twist $A'/K$ has $K$-period $N$ and $K$-parity $-1$.
Let $e_M={\rm ord}_2(M)$ and $e_N={\rm ord}_2(N)$.
If $e_N \leq e_M$, then ${\rm ord}_2(t) = 1 + e_M$;
if $e_N > e_M$, then ${\rm ord}_2(t) = e_N$.
\end{proposition}

\begin{proof}
Write $L={\rm lcm}(M,N)$.
Recall that $z_i^M=-1$ and $w_i^N = 1$ for $1 \leq i \leq g$.

Suppose that $e_N \leq e_M$.  Then $\ell_2=L/M$ is odd and ${\rm ord}_2(L)=e_M$.  
Then
\[1=w_i^L = \lambda_i^{L} z_i^L=  \lambda_i^{L} (z_i^M)^{\ell_2} = \lambda_i^{L} (-1)^{\ell_2}.\]
This implies that $\lambda_i^{L} = -1$ and so
${\rm ord}_2(o(\lambda_i) ) = 1 + e_M$ for $1 \leq i \leq g$.

Suppose that $e_M < e_N$.
For $1 \leq i \leq g$, then ${\rm ord}_2(o(z_i) ) = 1+e_M$ 
and ${\rm ord}_2(o(w_i) ) \leq e_N$.
The equation $w_i = \lambda_i z_i$ implies that 
${\rm ord}_2(o(\lambda_i) ) \leq e_N$ for $1 \leq i \leq g$.
To show that ${\rm ord}_2(t) = e_N$, it thus
suffices to show that ${\rm ord}_2(o(\lambda_i) ) = e_N$ for some $i$.

When $e_M < e_N$, then $rN/2$ is even, because $rM$ is even by definition of the period. 
So if $r$ is odd, then $e_N > e_M \geq 1$.
By the minimality of $N$ (such that $rN$ is even), 
it cannot hold that $w_i^{N/2} = 1$ for all $i$.
Thus, there is at least one value $i_0$ such that 
${\rm ord}_2(o(w_{i_0}) ) = e_N$.
Furthermore, since the $K$-parity is $-1$, 
it is not true that $w_i^{N/2} = -1$ for all $i$.
So there is at least one value $i_1$ such that 
${\rm ord}_2(o(w_{i_1}) ) < e_N$.

Note that $z_{i} = \lambda_{i}^{-1} w_{i}$.
If $e_N > 1+ e_M$, then substituting $i=i_0$ shows that 
${\rm ord}_2(o(\lambda_{i_0}^{-1}) ) = e_N$.  
If $e_N=1+e_M$, then substituting $i=i_1$ shows
${\rm ord}_2(o(\lambda_{i_1}^{-1}) ) = 1+e_M$.  
\end{proof}

\begin{corollary}\label{Ctwistodd}
\begin{enumerate}
\item Suppose that $A/K$ has $K$-period $M$ and $K$-parity $+1$.
If $A'/K$ is a $K$-twist of order $T$ with ${\rm ord}_2(T) \leq e_M$, 
then $A'/K$ also has $K$-parity $+1$.

\item Suppose that $A'/K$ has $K$-period $N$ and $K$-parity $-1$.
If $A/K$ is a twist of order $T$ with either ${\rm ord}_2(T) < e_N$
or ${\rm ord}_2(T) = e_N=0$, 
then $A/K$ also has $K$-parity $-1$.

\item In particular, if $A/K$ and $A'/K$ have different $K$-parities, then $T$ is even.
\end{enumerate}
\end{corollary}

\begin{proof}
Note that ${\rm ord}_2(T) \geq {\rm ord}_2(t)$.
\begin{enumerate}
\item Assume that $A'/K$ has parity $-1$. 
By Proposition \ref{Ptwistodd},
${\rm ord}_2(t) = 1+e_M$ if $e_N \leq e_M$ and ${\rm ord}_2(t) = e_N$ if $e_N > e_M$.
So ${\rm ord}_2(T) > e_M$, which is a contradiction.

\item Assume that $A/K$ has parity $1$.
Applying Proposition \ref{Ptwistodd} shows that 
${\rm ord}_2(t) = 1+e_M$ if $e_N \leq e_M$ and ${\rm ord}_2(t) = e_N$ if $e_N > e_M$.
This implies that either ${\rm ord}_2(T) \geq e_N$ or ${\rm ord}_2(T) > e_N=0$, which is a contradiction.

\item If $T$ is odd, then ${\rm ord}_2(T) =0$.
The hypotheses of items (1) and (2) are satisfied and so $A/K$ and $A'/K$ have the 
same parity.
\end{enumerate}
\end{proof}

\section{Fully maximal, fully minimal, and mixed curves} \label{S5}

Let $K = \FF_q$ with $q=p^r$ and let $k = \overline{\FF}_p$.
Let $X/K$ be a smooth projective connected supersingular curve of genus $g$.
The Jacobian $\mathrm{Jac}(X)$ of $X$
is a principally polarized abelian variety of dimension $g$.
If $X$ is hyperelliptic, let $\iota$ denote its hyperelliptic involution.

\subsection{Types for Jacobians}\mbox{}\\

The theory of twists of $X$ and definitions of the period and parity of $X$ are almost identical 
to those of ${\rm Jac}(X)$, as studied in Sections \ref{S3} and \ref{S4}.
The \NWN $\{z_i, \bar{z}_i\}_{1 \leq i \leq g}$ and the $2$-valuation vector
$\underline{e}=\{e_i ={\rm ord}_2(o(z_i))\}_{1 \leq i \leq g}$ 
are the same for $X$ and ${\rm Jac}(X)$. 
The main difference is that $X$ may have fewer twists than ${\rm Jac}(X)$.
Let $\Theta(X/K)$ denote the set of $K$-isomorphism classes of twists of $X/K$. 

\begin{definition}\label{Jactype}
A supersingular curve $X/K$
is of one of the following \emph{types} over K:
\begin{enumerate}
\item \emph{fully maximal} if $X'/K$ has $K$-parity $\delta=1$ for all $X' \in \Theta(X/K)$;
\item \emph{fully minimal} if $X'/K$ has $K$-parity $\delta=-1$ for all $X' \in \Theta(X/K)$;
\item \emph{mixed} if there exist $X', X'' \in \Theta(X/K)$ with $K$-parities $\delta(X')=1$ and $\delta(X'')=-1$.
\end{enumerate}
\end{definition}

By \cite[Appendice]{lauter}, ${\rm Jac}(X)$ has the same field of definition as $X$ and
\begin{equation}\label{algtor}
\mathrm{Aut}_k(\mathrm{Jac}(X)) \simeq
\begin{cases}
\mathrm{Aut}_k(X) \mbox{ if } X \mbox{ is hyperelliptic,}\\
\langle \iota \rangle \times \mathrm{Aut}_k(X) \mbox{ if } X \mbox{ is not hyperelliptic.}\\
\end{cases}
\end{equation}

When $X$ is hyperelliptic, then $\Theta(\mathrm{Jac}(X)/K) = \Theta(X/K)$, so $X$ and $\mathrm{Jac}(X)$ have the same type over $K$. When $X$ is not hyperelliptic, then $X$ and ${\rm Jac}(X)$ might have different types.

\begin{lemma}\label{xjac}
The types of $X$ and ${\rm Jac}(X)$ over $K$ are not the same if and only if:
$X$ is not hyperelliptic, ${\rm Jac}(X)$ is mixed over $K$, $r$ is even,
and $\underline{e}(X/K)=\{e\}$ with $e \leq 1$. 
\end{lemma}

\begin{proof}
If the types of $X$ and ${\rm Jac}(X)$ over $K$ are not the same, then ${\rm Jac}(X)$ has more twists than $X$, 
so \eqref{algtor} implies that $X$ is not hyperelliptic.
Also, since the extra twist corresponds to $\iota$, then
${\rm Jac}(X)$ is mixed, with ${\rm Jac}(X)$ and ${\rm Jac}(X)_\iota$ having different parities.

Let $\underline{e}=\underline{e}(X/K)$.
If not all $e_i \in \underline{e}$ are the same, 
then not all $e_i \in \underline{e}({\rm Jac}(X)_\iota)$ are the same.
Then both $\mathrm{Jac}(X)$ and ${\rm Jac}(X)_\iota$ would have parity $-1$, a 
contradiction.
Thus $\underline{e}=\{e\}$.

If $e \geq 2$, then $\underline{e}({\rm Jac}(X)_\iota) =\{e\}$ and 
both $\mathrm{Jac}(X)$ and ${\rm Jac}(X)_\iota$ would have parity $1$, a contradiction.
Thus $e \leq 1$ and $r$ must be even by \eqref{no01rodd}.
We omit the converse direction.
\end{proof}

The following results are immediate from Definition \ref{Jactype} and Proposition \ref{maxtomin}.

\begin{lemma}\label{Lbasicdef}
Suppose that $X$ has $K$-period $1$.  Then $X$ is mixed if and only if $X$ is hyperelliptic; $X$ is fully maximal if and only
it is not hyperelliptic and maximal; and $X$ is fully minimal if and only
it is not hyperelliptic and minimal.   
\end{lemma}

\begin{lemma} \label{Lbasicfact}
If ${\rm Aut}_k(X)$ is trivial, then $X$ is fully maximal over $K$ if and only if it has $K$-parity $1$ and is fully 
minimal if and only it has $K$-parity $-1$.
\end{lemma}

In light of Lemmas \ref{Lbasicdef} and \ref{Lbasicfact}, it is most interesting to study the types of curves which are 
non-hyperelliptic, defined over small fields, or have non-trivial automorphism group.

\subsection{Supersingular non-hyperelliptic curves of mixed type}\mbox{}\\

Despite Proposition \ref{maxtomin}, the results in Sections \ref{sec:low} and \ref{Ssurface}
show that not all hyperelliptic curves are mixed. The next result illustrates that not all mixed curves
are hyperelliptic.

\begin{proposition} \label{Pexamplemixed}
Suppose that $s \equiv 0 \bmod 4$.  Suppose that $p$ is such that $p+1\equiv 0 \bmod s$.
Then the smooth plane curve $X/\FF_p$ of genus $g=(s-1)(s-2)/2$ given by the equation $x^s+y^s+z^s=0$ is supersingular and of mixed type over $\FF_p$.
\end{proposition}

\begin{proof}
The curve $X/\FF_p$ is a smooth plane curve, of genus $g=(s-1)(s-2)/2$ by the Plucker formula.

The Hermitian curve $\tilde{X}: x_1^{p+1} + y_1^{p+1} + z_1^{p+1}=0$ is maximal over $\FF_{p^2}$. 
Let $\epsilon =(p+1)/s$.
There is a cover $\psi: \tilde{X} \to X$ given by $(x_1,y_1,z_1) \mapsto (x_1^\epsilon, y_1^\epsilon, z_1^\epsilon)$.
The cover is Galois, since there exists $\lambda \in \FF_{p^2}^*$ with multiplicative order $\epsilon$.
So $X$ is a quotient of $\tilde{X}$ by a group of automorphisms defined over $\FF_{p^2}$.
By a result attributed to Serre, see \cite[Theorem 10.2]{Kbook},
$X$ is also maximal over $\FF_{p^2}$ and thus has $\FF_{p}$-parity $1$.
In particular, $X$ is supersingular.

Let $\lambda_1 \in \FF_{p^2}^*$ be an element of multiplicative order $s_1=s/2$.  
Consider the automorphism $h \in {\rm Aut}_{\FF_{p^2}}(X)$ given by 
$h(x,y,z) = (\lambda_1 y, x, z)$.
Then 
\begin{eqnarray*}
h\prescript{{Fr_{\FF_p}}}{}{h} (x,y,z) & = & h (Fr_{\FF_p} ( h (x^{1/p}, y^{1/p}, z^{1/p}))) \\
& = & h (Fr_{\FF_p} (\lambda_1 y^{1/p}, x^{1/p}, z^{1/p})) = h (\lambda_1^p y, x, z)\\
& = & (\lambda_1 x, \lambda_1^p y, z) = (\lambda_1 x, \lambda_1^{-1} y, z),
\end{eqnarray*}
where the last equality uses that $p \equiv -1 \bmod s$.
In particular, $h\prescript{{Fr_{\FF_p}}}{}{h}$ has order $s_1$.

Consider the action of $h\prescript{{Fr_{\FF_p}}}{}{h}$ on ${\rm Jac}(X)/\FF_{p^2}$.
The next claim is that the eigenvalues for this action include both $1$ and a root of unity of order $s_1$.
To see this, it suffices to prove the same claim for the action on $H^1(X, {\mathcal O})$ 
(after lifting to characteristic $0$, 
using that ${\rm Jac}(X) \simeq H^0(X, \Omega^1)^*/H_1(X, \ZZ)$ and invoking Serre duality).
Now $H^1(X, {\mathcal O})$ has a basis given by the monomials 
$x^{-k_1}y^{-k_2}z^{-k_3}$ where $k_1,k_2,k_3 \in {\mathbb N}$
and $k_1+k_2+k_3 = r$.
Then $h\prescript{{Fr_{\FF_p}}}{}{h}$ acts via multiplication by 
$\lambda_1^{-k_1+k_2}$ on $x^{-k_1}y^{-k_2}z^{-k_3}$.
The claim follows by taking $(k_1,k_2)=(1,1)$ and $(k_1,k_2)=(1,2)$. 

The normalized Weil numbers of $X/\FF_{p^2}$ are all $-1$ and so $\underline{e}(X/\FF_{p^2}) = \{1\}$.
Let $X'$ be the twist of $X/\FF_{p}$ corresponding to $h$.
Then $X'/\FF_{p^2}$ is the twist of $X/\FF_{p^2}$ by $h\prescript{{Fr_{\FF_p}}}{}{h}$.
Its set of normalized Weil numbers contains $-1$ and $- \lambda_1$.
By hypothesis, $s_1$ is even.  So $- \lambda_1$ has odd order if $s_1  \equiv 2 \bmod 4$ and 
$- \lambda_1$ has order $s_1$ if $s_1 \equiv 0 \bmod 4$.
Thus $\underline{e}(X'/\FF_{p^2})$ contains the values $1$ and $0$ if $s_1 \equiv 2 \bmod 4$ and the values 
$1$ and ${\rm ord}_2(s_1) \geq 2$ if $s_1 \equiv 0 \bmod 4$.
In either case, $\underline{e}(X'/\FF_{p^2}) \not = \{e\}$ and $\underline{e}(X'/\FF_{p}) \not = \{e\}$ for any $e$.
Hence, $X'$ has $\FF_{p}$-parity $-1$.
Thus $X$ is mixed over $\FF_p$.
\end{proof}

\begin{example} \label{Efermatquartic}
For $p \equiv 3 \bmod 4$, the Fermat curve $X/\FF_p: x^4+y^4+z^4 = 0$ is a non-hyperelliptic supersingular curve of genus $3$ which is mixed over $\FF_p$.
\end{example}

\begin{remark}
Let $p$ be odd.
Let $E/\FF_p$ be a supersingular elliptic curve with $\NWNs(E/\FF_p)=\{i, -i\}$.
In \cite[Theorem 1]{ibu3} (resp.\ \cite[Proposition 15]{HLP}),
the authors construct a smooth plane quartic $X/\FF_p$ such that 
${\rm Jac}(X) \simeq_{\FF_{p^2}} E^3$ 
(resp.\ ${\rm Jac}(X) \sim_{\FF_{p^2}} E^3$).
In particular, $X$ is maximal over $\FF_{p^2}$.
The polarization on ${\rm Jac}(X)$ induces a non-product polarization on $E^3$.
To determine the type of $X$, it is necessary to determine which automorphisms of $E^3$ are compatible with this polarization and the field of definition of these automorphisms. 
\end{remark}

\subsection{Parity-changing twists of curves} \label{SsmallautX} \mbox{}\\

Let $X/K$ be a supersingular curve of genus $g$ and let $G={\rm Aut}_k(X)$.
The \NWN determine the $K$-parity of $X$.
To determine the type over $K$, it is necessary to know whether $X$ has a parity-changing $K$-twist.

By Corollary \ref{Ctwistodd}, the $2$-divisibility of the order $T$ of a twist gives information about whether it 
can change the $K$-parity of $X$.
However, this is not easy to control because the values of $T$ 
depend on the $K$-Frobenius conjugacy classes of $G$ and on the fields of definition of the automorphisms $g \in G$. 

This section contains results that simplify the question of whether $X$ has a parity-changing twist. 
This material is used in Section \ref{SVR}.
Given $g \in G$, recall from Proposition \ref{bij1} 
that $\phi_g: X \times_K k \to X' \times_K k$ is a geometric isomorphism such that $\xi_\phi(Fr_K)=g$.

\begin{lemma} \label{Loddnotparity}
If $h \in G$ has odd order and is defined over $K$, then $\phi_h$ is not a parity-changing twist. 
\end{lemma}

\begin{proof}
This is immediate from Lemmas \ref{LTcG} and \ref{Ctwistodd}. \end{proof}

Suppose that $\tau \in {\rm Aut}_k(X)$ has order $2$. 
Assume that $\tau$ is defined over $K$; this is true, for example, if $\tau=\iota$ or 
if ${\rm Aut}_k(X)$ has a unique element of order $2$.
Let $Z=X/\tau$ be the quotient of $X$ by $\tau$, which is also defined over $K$. 
Thus, $X \to Z$ is a geometric $\ZZ/2\ZZ$-Galois cover. 
Let $\chi$ be the nontrivial character of $\ZZ/2\ZZ$; it satisfies $\chi(P)=1$ if $P \in Z$ is split in $X$ and $\chi(P)=-1$ if $P$ is inert in $X$. Consider the Artin $L$-series
\begin{equation}\label{L2euler}
L(Z/K,T,\chi) = \prod_{P \in Z} (1 - \chi(P) \vert P\vert ^{-s})^{-1}.
\end{equation}

\begin{lemma} \label{eulergeneral}
Suppose that $\tau \in {\rm Aut}_K(X)$ has order $2$. 
\begin{enumerate}
\item There is a factorization $L(X/K, T)=L(Z/K, T) L(Z/K, T, \chi)$ in $\ZZ[T]$.
\item The coefficient $\rho_1$ of $T$ in $L(Z/K, T, \chi)$ equals $S_1-I_1$, 
where $I_1$ (resp.\ $S_1$) is the number of $K$-points of $Z$ that are 
inert (resp.\ split) in $X$. 
\item $\tau$ negates the roots of $L(Z/K, T, \chi)$ and fixes the roots of $L(Z/K, T)$.
\end{enumerate}
\end{lemma}

\begin{proof}
\begin{enumerate}
\item This result follows from \cite[Chapter 9, page 130]{rosen}. 
\item Recall that $\zeta(X/K, T)= \prod_{Q \in X} (1 - \vert Q\vert ^{-s})^{-1}$, where $T=q^{-s}$. Similarly, $\zeta(Z/K, T)= \prod_{P \in Z} (1 - \vert P\vert ^{-s})^{-1}$.
Write
\begin{equation}\label{isrproduct}
\zeta(X/K, T)= \prod (1-\vert P_i\vert ^{-2s})^{-1} \prod(1-\vert P_{sp}\vert ^{-s})^{-2} \prod(1-\vert P_r\vert ^{-s})^{-1},
\end{equation}
where $P_i$, $P_{sp}$, $P_r$ range over points of $Z$ that are inert, split, and ramified in $X$, respectively.
Note that $(1- \vert P\vert ^{-2s})=(1-\vert P\vert ^{-s})(1+\vert P\vert ^{-s})$.
The result follows by comparing \eqref{L2euler} and \eqref{isrproduct} and computing the coefficients of $T$.
\item Since $Z = X/\tau$, the involution $\tau$ acts trivially on $Z$ and thus 
fixes the roots of $L(Z/K, T)$.
There is an isogeny $\mathrm{Jac}(X) \sim_K \mathrm{Jac}(Z) \oplus V$ 
where $V/K$ is the nontrivial eigenspace for $\tau$.
Then $L(Z/K,T,\chi)=L(V/K,T)$.
By Proposition \ref{frobtwist}, $\tau$ acts as -1 on the roots of $L(V/K, T)$ by the definition of $V$.
\end{enumerate}
\end{proof}

Suppose that $\tau \in {\rm Aut}_K(X)$ has order $2$.
Write $\underline{e}=\underline{e}(Z/K) \cup \underline{e}(Z/K, \chi)$ where 
$\underline{e}(Z/K, \chi)$ denotes the multiset of $2$-valuations of the 
normalized roots of $L(Z/K,T,\chi)$.
If $\tau$ is the hyperelliptic involution, then $\underline{e}(Z/K)$ is empty and 
$\underline{e}= \underline{e}(Z/K, \chi)$.

\begin{lemma} \label{Levenparitychange}
If $\tau \in {\rm Aut}_K(X)$ has order $2$, then $\phi_\tau$ is a parity-changing twist if and only if 
$r$ is even and either $\underline{e}(Z/K)=\{1\}$ and $\underline{e}(Z/K, \chi)=\{e\}$ with $e \leq 1$,
or $\underline{e}(Z/K)=\{0\}$ and $\underline{e}(Z/K, \chi)=\emptyset$.
\end{lemma}

\begin{proof}
By Lemma \ref{eulergeneral}, $\tau$ negates the roots of $L(Z/K, T, \chi)$ and fixes the roots of $L(Z/K, T)$.
This changes the parity only under the given conditions. \end{proof}

Information about parity-changing twists can be determined from $\underline{e}$ in certain cases 
when $\ZZ/2\ZZ \times \ZZ/2\ZZ  \subset {\rm Aut}_k(X)$ 
using the next remark.  Section \ref{SVRfinal} uses this material.

\begin{remark} \label{R22twist}
Suppose that ${\rm Aut}_k(X)$ contains a subgroup $S \simeq \ZZ/2\ZZ \times \ZZ/2\ZZ$.  
Write $S=\{{\rm id}, \tau_1, \tau_2, \tau_3\}$.
Suppose that $S$ is stabilized by $K$-Frobenius conjugation, in which case
the number $\gamma$ of nontrivial involutions in $S$ defined over $K$ is either $3$, $0$, or $1$.

\begin{enumerate}
\item
When $\gamma=3$ and $X/S$ has genus $0$, let $A_i={\rm Jac}(X/\tau_i)$.
Then $\mathrm{Jac}(X) \sim_K A_1 \oplus A_2 \oplus A_3$ by \cite[Theorem B]{kanirosen}.  
Each $\tau_i$ acts by negating $\NWNs(A_i/K)$ for exactly two values of $i$. 
Write $\underline{e}_i = \underline{e}(A_i)$ and $\underline{e}(X) = \bigcup_{i=1}^{3} \underline{e}_i$. 
The twist for $\tau_i \in S$ changes the parity if and only if $\underline{e}(X) = \{1\}$ or  
(after rearranging), $\underline{e}_1 = \{1\}$, $\underline{e}_2=\{0\}$, and 
$\underline{e}_3= \{0\}$ or $\emptyset$. 

\item
When $\gamma = 0$, 
$K$-Frobenius conjugation acts via a $3$-cycle on $S-\{{\rm id}\}$, so the twist for each $\tau_i$ has order $3$.  
By Corollary \ref{Ctwistodd}, these do not change the parity.  

\item
When $\gamma=1$, 
suppose $\tau=\tau_1$ is defined over $K$ while $\mu=\tau_2$ and $\mu \tau =\tau_3$ are not.
Let $Z = X/\tau$. 
Note that $\prescript{{Fr_K}}{}{\mu}=\mu \tau$ and $\mu \prescript{{Fr_K}}{}{\mu} = \tau$. 
Using Lemma \ref{LTcG}, the twist for $\mu$ has $c=2$ and $\vert G\vert =2$.
Moreover, the twist by $\mu$ over $K$ corresponds to the twist $X_\tau$ by $\tau$ over $K_2$, 
so it negates the roots of $L(Z/K_2, T, \chi)$ and fixes the roots of $L(Z/K_2, T)$ by Lemma \ref{eulergeneral}(3).
To find the action of $\mu$ on $\underline{e}(X/K)$, it is necessary to take the square roots of the 
$\NWNs(X_\tau/K_2)$.
If $e_i \leq 1$ for any $i$, this leads to some ambiguity in $\underline{e}(X_\mu/K)$,
which can be partially resolved by the following observation.

\begin{equation} \label{Eambiguity}
{\rm {\bf Claim:} \ When\ } \gamma=1, {\rm \ the \ coefficient \ }\rho_1 {\rm \ of \ }
T {\rm \ in \ }L(Z/K, T, \chi) {\rm \ equals \ }0.
\end{equation}

\begin{proof}
By Lemma \ref{eulergeneral}(2), it suffices to prove $S_1=I_1$.
If $p$ is odd, $X \to Z$ has an equation of the form $y^2=F$.
Given a $K$-point $P$ of $Z$, it suffices to show $P$ is split in $X$ if and only if $\mu(P)$ is inert in $X$.
The point $P$ splits in $X$ if and only if $F(P)$ is a square in $K^*$.
Since $\mu$ and $\tau$ commute, $\mu$ acts on both $X$ and $Z$. 
By assumption, 
the action of $\mu$ on the equation $y^2=F$ is defined over $K_2$ but not over $K$.
The $K$-action of $\mu$ thus yields a quadratic twist of $y^2=F$.  So $\mu(y)=w y$ 
for some $w \in K_2^* \setminus K^*$ such that $z=w^2$ is in $K^*$, and $F(\mu(P)) = z y$.
Thus, $F(P)$ is a square in $K^*$ if and only if $F(\mu(P))$ is not.

The proof for $p=2$ is the same, after replacing $y^2$ by $y^2-y$, $\mu(y)=wy$ by $\mu(y)=y+w$ for some $w \in K_2 \setminus K$
such that $z=w^2-w$ is in $K$, and $F(\mu(P)) = zF(P)$ by $F(\mu(P)) = F(P) + z$. \end{proof}
\end{enumerate}
\end{remark}


\section{Analysis in low dimension: elliptic curves}\label{sec:low}

Let $K = \FF_{q}$ with $q=p^r$ and let $k = \overline{\FF}_p$.
If $E/\FF_q$ is an elliptic curve, then
$L(E/\FF_{q},T) = 1 - \beta T + qT^2$ for some $\beta \in \ZZ$.
Moreover, $E$ is supersingular if and only if $p \mid \beta$.
By Honda-Tate theory (cf.\ \cite{tate71}, \cite{honda}, \cite{tate}), $\beta$ determines 
the $\FF_{q}$-isogeny class of $E$.  

\begin{lemma}\label{g=1weilnumbers}
Let $q=p^r$. 
The following table lists each $\beta \in \ZZ$ which occurs for 
a supersingular elliptic curve $E/\FF_q$, 
together with the normalized Weil numbers $z$ and $\bar z$, the $2$-adic valuation $e={\rm ord}_2(o(z))$, 
the period, and the parity. We use the convention that $\zeta_n = e^{2\pi i/n}$.

\begin{longtable}{| l | p{4cm} | r | l | l | l | l | }	\caption{Supersingular elliptic curves.}\label{tab:g=1}\\   
 \hline
  Case  $n_E$ & Conditions on $r$ and $p$ & $\beta$ & $\NWNs(E/\FF_q)$ & ${\rm ord}_2(o(z) )$ & Period & Parity \\ \hline
$W1 \pm$ & $r$ even & $\pm 2\sqrt{q}$ & $(\pm 1, \pm 1)$ & 0 & 1 & $\mp 1$ \\
$W2 \pm$ & $r$ even, $p \not\equiv 1 \pmod 3$ & $\pm \sqrt{q}$ & $(\mp \zeta_3, \mp \overline{\zeta}_3)$ & 1 & 3 & $\pm 1$ \\
$W3$ & $r$ even, $p \not \equiv 1 \pmod 4$ or $r$ odd & 0 & $(i,-i)$ & 2 & 2 & 1 \\
$W4a$ & $r$ odd, $p = 2$ & $ \pm \sqrt{2q}$ & $(\pm \zeta_8, \pm \overline{\zeta}_8)$ & 3 & 4 & 1 \\
$W4b$ & $r$ odd, $p = 3$ & $\pm \sqrt{3q}$ & $(\pm \zeta_{12}, \pm \overline{\zeta}_{12})$ & 2 & 6 & 1 \\
	\hline
	\end{longtable}
\end{lemma}

\begin{proof}
This is a short calculation based on the values of $\beta$ in \cite[Theorem 4.1] {water}.
\end{proof}

The number of supersingular $j$-invariants is $\lfloor \frac{p}{12}\rfloor + \epsilon$ (with $\epsilon = 0,1,1,2$ if $p \equiv 1,5,7,11 \bmod 12$) \cite[Theorem V.4.1(c)]{aec}.

\begin{remark} \label{Ranalytic}
Let $N(\beta)$ denote the number of $\FF_{q}$-isomorphism classes of elliptic curves 
in the $\FF_{q}$-isogeny class determined by $\beta$.
The values of $N(\beta)$ are found in \cite[Theorem 4.6]{schoof}; they depend only on $p$, not $q$, 
and $N(-\beta) = N(\beta)$.
Using this and Table \ref{tab:g=1}, one can determine the probability that a given supersingular elliptic curve
$E/\FF_{p^r}$ has $\FF_{p^r}$-parity $1$.
If $r$ is odd, then the $\FF_{p^r}$-parity is always $1$.
If $r$ is even, then $N(0)=1- \binom{-4}{p}$ is
the difference between the number of isomorphism classes of $E/\FF_{p^r}$
with $\FF_{p^r}$-parity $1$ and $-1$.
\end{remark}

Each supersingular $j$-invariant is in $\FF_{p^2}$. 
If $E/\overline{\FF}_p$ is a supersingular elliptic curve, then $E$ descends to $\FF_p$ or $\FF_{p^2}$;
it descends to $\FF_p$ if and only if the $j$-invariant of $E$ is in $\FF_p$.  
The next result shows that in neither case is $E$ fully minimal.

\begin{theorem} \label{Tconcg1}
Let $E/\overline{\FF}_p$ be a supersingular elliptic curve.
If the $j$-invariant of $E$ is in $\FF_p$,
then $E$ is fully maximal over $\FF_p$; 
if not, then $E$ is mixed over $\FF_{p^2}$.
\end{theorem}

\begin{proof} 
If $p=2$, the result is proven in Lemma \ref{Lg1p23} (below).
If $p \geq 3$ and ${\rm Aut}_k(E) \not \simeq \ZZ/2\ZZ$, the result is proven in Lemma \ref{LEextraaut} (below).  
This completes the proof for $p = 3$, since there is only one
isomorphism class of supersingular elliptic curves over $\overline{\FF}_3$.

Finally, suppose that $p \geq 5$ and ${\rm Aut}_k(E) \simeq \ZZ/2\ZZ$, so that $E_{\iota}$ is the only twist of $E$.
If $E$ is defined over $\FF_p$, then $E$ and $E_{\iota}$ are both in case W3 of Table \ref{tab:g=1}, thus $E$ is fully maximal over $\FF_p$.
If $E$ is instead defined over $\FF_{p^2}$, then $E$ and $E_{\iota}$ are either in cases W1$\pm$ or 
in cases W2$\pm$ of Table \ref{tab:g=1};
note that $E$ cannot be in case W3 because of the condition ${\rm Aut}_k(E) \simeq \ZZ/2\ZZ$
(and in that case $E$ has $j$-invariant in $\FF_p$).
Thus $E$ is mixed over $\FF_{p^2}$.
\end{proof}

\begin{lemma} \label{Lg1p23}
If $p=2$, the unique supersingular elliptic curve $E/\overline{\FF}_2$ is fully maximal over~$\FF_2$.
\end{lemma}

\begin{proof}
The uniqueness fact can be found in \cite[Appendix A, Proposition 1.1]{aec}.
So $E$ is isomorphic over $k$ to the elliptic curve $E/\FF_2$ with affine equation $y^2=x^3-x$ with $j$-invariant $0$.
Then $\vert E(\FF_2)\vert =p+1$, so $\beta=0$ (case W3 of Table \ref{tab:g=1}).
The $\FF_{2}$-twists are also defined over $\FF_{2}$, thus are in case W3, W4a or W4b of Table \ref{tab:g=1}, 
which each have $\FF_{2}$-parity $+1$.
\end{proof}

\begin{lemma} \label{LEextraaut}
Let $p \geq 3$.  If ${\rm Aut}_k(E) \not \simeq \ZZ/2\ZZ$, 
then $E$ is fully maximal over $\FF_p$.
\end{lemma}

\begin{proof}
If ${\rm Aut}_k(E) \not \simeq \ZZ/2\ZZ$, then 
$E$ is isomorphic over $k$ to either:
\begin{enumerate}
\item $y^2=x^3-x$
($j$-invariant $1728$), which is supersingular if and only if $p \equiv 3 \bmod 4$; or
\item $y^2=x^3+1$, 
($j$-invariant $0$), which is supersingular if and only if $p \equiv 2 \bmod 3$. 
\end{enumerate}
In both cases, $\{z, \bar{z}\}=\{i, -i\}$ (case W3 of Table \ref{tab:g=1})
with $\underline{e}(E/\FF_p)=\{2\}$ and the curve is defined over $\FF_p$, so we consider its type over $\FF_p$.

For case (1), 
let $g \in \mathrm{Aut}_k(E)$ be the order $4$ automorphism
defined by $g(x,y)=(-x, iy)$.
\begin{itemize}
\item[(a)] If $p>3$, then $\mathrm{Aut}_k(E) \simeq \langle g \rangle$.
Then $E/\FF_p$ has only one nontrivial twist because the $\FF_p$-Frobenius conjugacy classes in $\mathrm{Aut}_k(E)$ are $\{{\rm id}, \iota\}$ and $\{g, g^3\}$.  By Lemma \ref{LTcG}, 
the latter of these yields a quadratic twist since $c=2$ and $G=g \prescript{{Fr}}{}{g}={\rm id}$.
By Lemma \ref{Ltwistsimple}, the twist has $\underline{e}=\{2\}$ as well.

\item[(b)] If $p=3$, then $\vert \mathrm{Aut}_k(E) \vert = 12$ \cite[Appendix A, Proposition 1.2]{aec}. Then $\mathrm{Aut}_k(A) = \langle g, \sigma \rangle$ where $\sigma(x,y)=(x+1, y)$.
The $\FF_{p}$-Frobenius conjugacy classes are $\{ \mathrm{id}, \iota \}$, $\{ \sigma^2, \sigma\iota \}$, $\{ \sigma, \sigma^2\iota, \}$, and $\{ g, g^3, \sigma g, \sigma g^3, \sigma^2 g, \sigma^2 g^3 \}$.
The first (resp.\ last) of these yield a trivial (resp.\ quadratic) twist as in (a).  Since $\sigma$ and $\sigma^2$ have order $3$ and are defined over $\FF_p$, 
these yield twists of order $3$ by Lemma \ref{LTcG} 
with $\underline{e}=\{2\}$ by Lemma \ref{Ltwistsimple}.

\end{itemize}

For case (2), $\mathrm{Aut}_k(E) = \langle h \rangle$ where  
$h$ has order $6$ and is
defined by $h(x,y)=( \zeta_3 x, -y)$.  The two $\FF_{p}$-Frobenius conjugacy classes are $\{\mathrm{id}, h^2, h^4\}$ and $\{h, h^3, h^5\}$.
Since $h^3 = \iota$, the latter of these yields a quadratic twist.
By Lemma \ref{Ltwistsimple}, the twist has $\underline{e}=\{2\}$ as well.

Thus, in both case (1) and case (2), $E$ is fully maximal over $\FF_p$.
\end{proof}

\section{Analysis in low dimension: abelian surfaces} \label{Ssurface}

\subsection{Parity table for simple abelian surfaces}\mbox{}\\

Let $q=p^r$ and $k = \overline{\FF}_p$. 
Suppose that $A/\FF_q$ is an abelian surface, which is not necessarily principally polarized. 
The $\FF_{q}$-isogeny class of $A$ is determined by (the conjugacy class of) its Weil numbers or, 
equivalently, by the coefficients $(a_1,a_2)$ of
\[P(A/\FF_{q},T) = T^4 + a_1T^3 + a_2T^2 + qa_1T + q^2 \in \ZZ[T].\]


The next result builds on \cite{maisnart}.
Let $L$ be the minimal field extension of $\FF_q$ over which $A$ is not simple. 
Then $A \sim_L E \oplus E$, where $E/L$ is a supersingular elliptic curve. 

\begin{proposition}\label{g=2weilnumbers}
The following table classifies all $(a_1,a_2)$ which occur 
as the coefficients of $P(A/\FF_q, T)$ for a simple supersingular abelian surface $A/\FF_q$, 
together with the data:
\begin{itemize}
\item $t_0={\rm deg}(L/\FF_q)$;
\item $W$, labeling $E/L$ as in the first column of Table \ref{tab:g=1};
\item $z/L$, one of the normalized Weil numbers $(z,\overline{z},z,\overline{z})$ of $A/L$ (again $\zeta_n = e^{2\pi i/n}$); 
\item $\NWNs(A/\FF_q)$, the \NWN of $A/\FF_{q}$;
\item $\mu$ and $\delta$, the period and parity respectively of $A/\FF_q$.
\end{itemize}
\begin{center}
\begin{table}
\caption{Simple supersingular abelian surfaces.}
\label{tab:g=2}
    \begin{tabular}{| l | l | p{5.2cm} | l | l | l | l | l | r |}    \hline
  & $(a_1,a_2)$ & Conditions on $r$ and $p$ & $t_0$ & $W$ & $z/L$ & $\NWNs(A/\FF_{q})$ & $\mu$ & $\delta$ \\ \hline
1a & $(0,0)$ & $r$ odd, $p \equiv 3 \bmod 4$ or $r$ even, $p \not \equiv 1 \bmod 4$ & $2$ & $3$ & $i$ & $(\zeta_8,\zeta^7_8, \zeta^3_8,\zeta^5_8)$ & $4$ & $1$ \\
1b & $(0,0)$ & $r$ odd, $p \equiv 1 \bmod 4$ or $r$ even, $p \equiv 5 \bmod 8$ & $4$ & $1$ & $-1$ & $(\zeta_8,\zeta^7_8, \zeta^3_8,\zeta^5_8)$ & $4$ & $1$ \\
2a & $(0,q)$ & $r$ odd, $p \not\equiv 1 \bmod 3$ & $2$ & $2$ & $\zeta_3$ & $(\zeta_6,\zeta_6^5, \zeta_6^2, \zeta_6^4)$ & $6$ & $-1$ \\
2b & $(0,q)$ & $r$ odd, $p \equiv 1 \bmod 3$ & $6$ & $1$ & $-1$ & $(\zeta_{12}, \zeta_{12}^{11}, \zeta_{12}^5, \zeta_{12}^7)$ & $6$ & $1$ \\
3a & $(0,-q)$ & $r$ odd and $p \equiv 2 \bmod 3$ or $r$ even and $p \not\equiv 1 \bmod 3$ & $2$ & $2$ & $-\zeta_3$ & $(\zeta_{12}, \zeta_{12}^{11}, \zeta_{12}^5, \zeta_{12}^7)$ & $6$ & $1$ \\
3b & $(0,-q)$ & $r$ odd and $p \equiv 1 \bmod 3$ or $r$ even and $p \equiv 4,7,10 \bmod{12}$ & $3$ & $3$ & $i$ & $(\zeta_{12}, \zeta_{12}^{11}, \zeta_{12}^5, \zeta_{12}^7)$ & $6$ & $1$ \\
4a & $(\sqrt{q},q)$ & $r$ even and $p \not\equiv 1 \bmod 5$ & $5$ & $1$ & $1$ & $(\zeta_5,\zeta_5^4,\zeta_5^2,\zeta_5^3)$ & $5$ & $-1$ \\
4b & $(-\sqrt{q},q)$ & $r$ even and $p \not\equiv 1 \bmod 5$ & $5$ & $1$ & $-1$ & $(\zeta_{10},\zeta_{10}^9,\zeta_{10}^3,\zeta_{10}^7)$ & $5$ & $1$ \\
5a & $(\sqrt{5q},3q)$ & $r$ odd and $p = 5$ & $10$ & $1$ & $1$ & $(\zeta_{10}^3,\zeta_{10}^7,\zeta_5^2,\zeta_5^3)$ & $10$ & $-1$ \\
5b & $(-\sqrt{5q},3q)$ & $r$ odd and $p = 5$ & $10$ & $1$ & $1$ & $(\zeta_{10},\zeta_{10}^9,\zeta_5,\zeta_5^4)$ & $10$ & $-1$ \\
6a & $(\sqrt{2q},q)$ & $r$ odd and $p = 2$ & $4$ & $2$ & $-\zeta_3$ & $(\zeta_{24}^{13},\zeta_{24}^{11},\zeta_{24}^{19},\zeta_{24}^5)$ & $12$ & $1$ \\
6b & $(-\sqrt{2q},q)$ & $r$ odd and $p = 2$ & $4$ & $2$ & $-\zeta_3$ & $(\zeta_{24},\zeta_{24}^{23},\zeta_{24}^{7},\zeta_{24}^{17})$ & $12$ & $1$ \\
7a & $(0,-2q)$ & $r$ odd & $2$ & $1$ & $1$ & $(1,1,-1-1)$ & $2$ & $-1$ \\
7b & $(0,2q)$ & $r$ even and $p \equiv 1 \bmod 4$ & $2$ & $2$ & $-1$ & $(i,-i,i,-i)$ & $2$ & $1$ \\
8a & $(2\sqrt{q},3q)$ & $r$ even and $p \equiv 1 \bmod 3$ & $3$ & $1$ & $1$ & $(\zeta_3,\zeta_3^2,\zeta_3,\zeta_3^2)$ & $3$ & $-1$ \\
8b & $(-2\sqrt{q},3q)$ & $r$ even and $p \equiv 1 \bmod 3$ & $3$ & $1$ & $-1$ & $(\zeta_6,\zeta_6^5, \zeta_6, \zeta_6^5)$ & $3$ & $1$ \\
	\hline
	\end{tabular}
	\end{table}
\end{center}

\end{proposition}

\begin{proof} 
The list of $(a_1,a_2)$, conditions on $r$ and $p$, and $t_0$ are found in \cite[Table 1, page 325]{maisnart}.\footnote{We would like to thank a referee for pointing out that the value of $t_0$ in Case $5$ is incorrect in \cite{maisnart}.}
Applying \cite[Lemma 2.13, Theorem 2.9]{maisnart}, we
compute the coefficients of $P(A/L, T)$ where $L = \FF_{q^{t_0}}$
and determine $W$.
Then the values of $z/L$, the period, and the parity can be found using Table \ref{tab:g=1}.
The period is the product of $t_0$ and the period of $E$ over $\FF_{q}$
and the parities of $A$ and $E$ are the same.
To determine $\NWNs(A/\FF_{q})$, 
we solve $P(A/\FF_{q},T) = 0$ directly.
\end{proof}

We now give a full classification of the types of supersingular simple principally polarized abelian surfaces
with $\mathrm{Aut}_k(A) \simeq \ZZ/2\ZZ$, using Proposition \ref{g=2weilnumbers}.

\begin{proposition}\label{Pmain7}
 Let $A$ be a supersingular simple principally polarized abelian surface defined over $K = \FF_q$.
Assume that $\mathrm{Aut}_k(A) \simeq \ZZ/2\ZZ$.  In Proposition \ref{g=2weilnumbers}:
\begin{enumerate}
\item if $r$ is odd, then $A/K$ is not mixed;
cases $(1)$, $(2b)$, $(3a)$, $(6)$ are fully maximal and cases $(2a)$, $(5)$, $(7a)$ are fully minimal.
\item if $r$ is even, then $A/K$ is not fully minimal; 
cases $(1)$, $(3a)$, and $(7b)$ are fully maximal and cases $(4)$ and $(8)$ are mixed.
\end{enumerate}
\end{proposition}

\begin{proof} 
By \cite[Theorem 1]{HMNR}, the principal polarization restriction excludes exactly case $(3b)$. 
Since $\mathrm{Aut}_k(A) \simeq \ZZ/2\ZZ$, the type 
of $A$ over $K$ is determined from $\underbar{e}(A/K)$ by Corollary \ref{typetoei2}. 
This can be computed from the \NWN found in Proposition \ref{g=2weilnumbers}.
\end{proof}

\begin{remark}
The sizes of the isogeny classes listed in Table \ref{tab:g=2} are not known.
From \cite{XYY}, one could conjecture that a supersingular abelian surface over $\FF_q$ most likely has mixed type.
\end{remark}

\subsection{Curves of genus $2$ with extra automorphisms}\mbox{}\\

By \cite{Igusa}, there are six equations that describe all genus $2$ curves $X/K$
such that ${\rm Aut}_k(X) \not \simeq \ZZ/2\ZZ$.  
The number of $k$-isomorphism classes of these $X/K$ which are supersingular 
is known \cite[Theorem 3.3]{IKO}. 
The twists of $X/K$ are studied
in \cite{cardona} and \cite{CN}.  
We determine the type for all supersingular genus $2$ curves $X$ with 
${\rm Aut}_k(X) \not \simeq \ZZ/2\ZZ$, over the smallest field $K = \FF_q$ containing the coefficients of 
their defining equation.
Let $\vert \Theta \vert$ denote the number of $K$-twists of $X$.
We first analyze the three equations which have no moduli parameters.

\begin{proposition}\label{prop:easy}
Let $p > 5$.
The types over $\FF_p$ of the following genus $2$ curves $X/\FF_p$ with ${\rm Aut}_k(X) \not \simeq \ZZ/2\ZZ$, which are 
supersingular under the listed condition on $p$, are as follows.
\begin{center}
    \begin{tabular}{| l | l | l | l | l | l |}
    \hline
 & Equation  & Condition & $\mathrm{Aut}_k(X)$ & $\vert \Theta\vert$ & Type \\  \hline
$1$ & $y^2 = x^5 - 1$ & $p \not \equiv 1 \bmod 5$ & $\ZZ/10\ZZ$ & $2$ & fully maximal \\ 
$2$ & $y^2 = x^6 - 1$ & $p \equiv 2 \bmod 3$ & $2D_{12}$ & $7$ & mixed \\ 
$3$ & $y^2 = x^5 - x$ & $p \equiv 5,7 \bmod 8$ & $\tilde{S}_4$ & $6$ & mixed \\ \hline
\end{tabular}
\end{center}
Here $D_n$ is the dihedral group of order $n$
and $\tilde{S}_4$ is a $2$-covering of $S_4$.
\end{proposition}

\begin{proof}
The equations and automorphism groups are found in \cite[Theorem 3.1]{CN}.
The supersingular condition is found in \cite[1.11-1.13]{IKO}.
For equation (1), $\vert \Theta\vert =2$ by \cite[Proposition 11]{cardona}.
For equation (2), when $p \equiv 2 \bmod 3$, then $-3 \not \in (\FF_p^*)^2$, 
so $\vert \Theta\vert =7$ by \cite[Proposition 16]{cardona}.
For equation (3), when $p \equiv 5,7 \bmod 8$, then $-2 \not \in (\FF_p^*)^2$, 
so $\vert \Theta\vert =6$ by \cite[Proposition 17]{cardona}.

The pairs $(a_1,a_2)$ which occur for the twists of $X$ are in \cite[Sections 3.1-3.3, Tables 5,9,6,7]{CN}.
If $(a_1,a_2)=(0,2p)$, note that $\mathrm{Jac}(X) \sim_{\FF_p} E \oplus E$ where 
$E/\FF_p$ is in case $W3$ of Lemma \ref{g=1weilnumbers}, which has parity $1$.
Also, $(a_1,a_2)=(0,-2p)$ has parity $-1$ by case $(7a)$ of Proposition \ref{g=2weilnumbers}.

\begin{enumerate}
\item 
When $p \equiv 2,3 \bmod 5$, then $(a_1,a_2)=(0,0)$ for $X$ and $X_\iota$;
thus $X$ is fully maximal.
When $p \equiv -1 \bmod 5$, then $(a_1,a_2)=(0,2p)$ for $X$ and $X_\iota$;
thus $X$ is fully maximal.

\item 
When $p \equiv 2 \bmod 3$, let $\epsilon = (-1/p)$. 
The first two rows of \cite[Table 9]{CN} show that the parity $1$ case $(a_1,a_2)=(0,2p)$ occurs for $X$ or one of its $\FF_p$-twists, 
regardless of the value of $\epsilon$.
The third and fourth lines of \cite[Table 9]{CN} show that the parity $-1$ case $(a_1,a_2)=(0, -2p)$ 
occurs for $X$ or one of its $\FF_p$-twists, 
regardless of the value of $\epsilon$, as long as there exists $t \in \FF_p$ such that $t^2+4$ is not a square in $\FF_p^*$;
the existence of such a $t$ can be verified using a Jacobi sum argument.
So $X$ is mixed.

\item
If $p \equiv 5,7 \bmod 8$, then 
both $(0,2p)$ and $(0,-2p)$ occur as $(a_1,a_2)$ among
the twists of $X$, so $X$ is mixed. 
\end{enumerate} 
\end{proof}

Next, we analyze the three equations with moduli parameters. 
 
\begin{proposition}\label{hard}
Let $p > 5$. 
Any genus $2$ curve $X/\FF_q$ with ${\rm Aut}_k(X) \not \simeq \ZZ/2\ZZ$ is isomorphic over $k$ to one of equations (1)-(3) in Proposition \ref{prop:easy} or one of equations (4)-(6) below:
\begin{itemize}
\item[(4)] $y^2 = x^6 + ax^4 + bx^2 + 1$ where $a, b \in k$ are chosen such that $P(c,d) \neq 0$, where $c = ab$, $d = a^3 + b^3$, and $P(c,d) = (4c^3-d^2)(c^2-4d+18c-27)(c^2-4d-110c+1125)$;
\item[(5)] $y^2 = x^5 + x^3 + ax$, for $a \neq 0$, $1/4$, $9/100$;
\item[(6)] $y^2 = x^6 + x^3 + a$ for $p \neq 3$, $a \neq 0$, $1/4$, $-1/50$.
\end{itemize}
Let $q=p^r$ be such that $a,b \in K=\FF_{q}$.  The types over $\FF_q$ for equations (4)-(6) are as follows:
\begin{center}
    \begin{tabular}{| l | l | l | l |}
    \hline
 & $\mathrm{Aut}_k(X)$ & $\vert \Theta\vert$ & Type \\  \hline
$4$ & $V_4$ & $4$ & $\begin{cases} \text{fully maximal} & \mbox{ if } r \mbox{ is odd } \\ \text{mixed} & \mbox{ if } r \mbox{ is even } \end{cases}$ \\ \hline
$5$ & $D_8$ & $3$ or $5$  & $\begin{cases} \text{fully maximal} & \mbox{ if } r \mbox{ odd, } a \notin (K^*)^2 \\ \text{mixed} & \mbox{ otherwise } \end{cases}$ \\ \hline
$6$ & $D_{12}$ & $4$ or $6$  & $\begin{cases} \text{fully maximal} & \mbox{ if } 
q \equiv 2 \bmod 3 
\mbox{ and } a \in (K^*)^2  \\ \text{mixed} & \mbox{ otherwise } \end{cases}$ \\ \hline
\end{tabular}
\end{center}
\end{proposition}

\begin{proof}
The equations and automorphism groups can be found in \cite[Theorem 3.1]{CN}. 
In cases (5) and (6), the number $\vert \Theta\vert $ of twists of $X$ is determined in \cite[Propositions 12-13]{cardona}.
In case (4), by \cite[Section 3.6]{CN}, when $X$ is supersingular, then $\vert \Theta\vert =4$. 
The pairs $(a_1,a_2)$ for the twists of $X$ are in \cite[Sections 3.4-3.6, Tables 11-17]{CN}.  
We determine the types over $\FF_{q}$ below:
\begin{itemize}
\item[(4)]
Since $\mathrm{Jac}(X) \sim_k E_1 \oplus E_2$, the 4 twists
of $X$ correspond to quadratic twists of either $E_1$ or $E_2$, or both. When $r$ is odd, $E_1$ and $E_2$ are both in case $W3$ of Lemma \ref{g=1weilnumbers}, so $X$ is fully maximal. 
When $r$ is even, $E_1$ and $E_2$ are either both in case $W1+$ (so $X$ is minimal) or both in case $W1-$ (so $X$ is maximal), depending on the $L$-polynomial of $E_1$.
Then $X$ is mixed since the quadratic twist swaps the two cases. 

\item[(5)] 
When $r$ is odd and $a \not\in (K^{\ast})^2$, then $X$ and its twists have $(a_1,a_2)$ equal to $(0,0)$ or $(0,2q)$. 
Since both cases have parity $1$, the curve $X$ is fully maximal.

When $r$ is odd and $a \in (K^{\ast})^2$, there are twists of $X$ with $(a_1,a_2)$ being both $(0,2q)$ (parity $1$) and $(0,-2q)$ (parity $-1$), 
so $X$ is mixed.
When $r$ is even, a similar argument shows that $X$ is mixed. 

\item[(6)]
When $q \equiv 2 \bmod 3$, note that $p \equiv 2 \bmod 3$ as well and $r$ is odd. 
Then $X$ and its twists have $(a_1,a_2)$ among $(0,2q)$, $(0,2\epsilon q)$ and $(0,-\epsilon q)$, where $\epsilon = 1$ if $a \in (K^*)^2$ and $\epsilon = -1$ otherwise. These curves have respective parities $1$, $\epsilon$, and $\epsilon$. So if $\epsilon = 1$, then $X$ is fully maximal and if $\epsilon = -1$, then $X$ is mixed. 

When $q \equiv 1 \bmod 3$ and $r$ is odd, then the coefficients $(a_1,a_2)$ of the twists 
include $(0,2q)$ and $(0,q)$ of parity $1$ and $(0,-2q)$ of parity $-1$, so $X$ is mixed.

When $q \equiv 1 \bmod 3$ and $r$ is even, let $\epsilon=\left(\frac{-3}{\sqrt{q}}\right)$. 
Then the possibilities for $(a_1,a_2)$ are $(\pm 4\epsilon\sqrt{q},6q)$ of parity $\pm\epsilon$, $(\pm2\epsilon\sqrt{q},3q)$ of parity $\mp\epsilon$, and $(0,-2q)$ of parity $-1$. So $X$ is mixed.
\end{itemize}
\end{proof}

\subsection{The condition $\mathrm{Aut}_k(A) \simeq \ZZ/2\ZZ$ is not restrictive when $p$ is odd}\mbox{}\\

For general $p$, $r$, and $g$, the structure of the typical automorphism group of a $g$-dimensional supersingular abelian variety $A$ over $K = \FF_{p^r}$ is unknown (cf.\ Remark \ref{Ronemoduli}). In this section, we resolve this question for $g=2$ and $p$ odd.

Let $g=2$ and let $A = (A,\lambda)$ be a principally polarized abelian surface. 
For $p \geq 3$, we prove that the proportion of $A$ over $\FF_{p^r}$ with 
${\rm Aut}_k(A) \not \simeq \ZZ/2\ZZ$ tends to zero as $r \to \infty$. 

Let ${\mathcal A}_2 = {\mathcal A}_2 \otimes \FF_p$ denote the moduli space whose points 
represent the objects $(A, \lambda)$ in characteristic $p$.
Let ${\mathcal A}_{2,ss} \subset {\mathcal A}_2$ denote the supersingular locus 
whose points represent supersingular $A$. 
Recall that $A$ is superspecial if and only if $A \simeq_k E_1 \oplus E_2$.

\begin{proposition}\label{Pmaincor}
If $p \geq 3$,
then the proportion of $\FF_{p^r}$-points in $\mathcal{A}_{2,ss}$ 
which represent
$A$ with ${\rm Aut}_k(A) \not \simeq \ZZ/2\ZZ$ tends to zero as $r \to \infty$. 
\end{proposition}

\begin{proof}
As observed in \cite[Section 9]{achterhowe},
$\vert  \mathcal{A}_{2,ss}(\FF_{p^r})\vert  \ll p^{r+2}$, 
where the notation $f(q) \ll g(q)$ means that there is a constant $C >0$ 
such that $\vert f(q)\vert  \leq C\vert g(q)\vert $ for all sufficiently large $q$.
This is because
each irreducible component of ${\mathcal A}_{2,ss}$ is geometrically isomorphic to $\mathbb{P}^1$ \cite[proof of Corollary 4.7]{oortsub}, and the number of irreducible components of $\mathcal{A}_{2,ss}$ equals the class number $H_2(1,p)$ 
\cite[Theorem 5.7]{KO}, which is $\ll p^2$ by \cite{HashI}, 
see also \cite[Remark 2.17]{IKO}.

By \cite[Theorem 3.1]{GGR}, an $\FF_{p^r}$-point $A$ in $\mathcal{A}_{2,ss}$ is one of the following canonically principally polarized objects:
(i) the Jacobian of a smooth supersingular curve $X$ over $\FF_{p^r}$ of genus $2$; 
(ii) the sum $E_1 \oplus E_2$ of two supersingular elliptic curves over $\FF_{p^r}$;
(iii) the restriction of scalars $\mathrm{Res}_{\FF_{p^{2r}}/\FF_{p^r}}(E)$ of a
supersingular elliptic curve $E/\FF_{p^{2r}}$.   
By \cite[Section 9]{achterhowe}, the number of objects in cases (ii) and (iii) is 
$\ll p^2$.

Thus, it suffices to restrict to case (i). 
Since $X$ is hyperelliptic, the isomorphism $A \cong_k \mathrm{Jac}(X)$ descends to $\FF_{p^r}$
by \cite[Appendix]{lauter}.
By \eqref{algtor},
${\rm Aut}_k({\rm Jac}(X)) \simeq {\rm Aut}_k(X)$. 
The arithmetic Torelli map is injective on $\FF_{p^r}$-points representing smooth curves
\cite[Corollary 12.2]{milneJac}.
So for case (i), it suffices to bound the number of supersingular curves $X$ of genus $2$ with ${\rm Aut}_k(X) \not \simeq \ZZ/2\ZZ$, which are described in cases (1)-(6) of 
Propositions \ref{prop:easy} and \ref{hard} when $p>5$; the cases $p=3$ and $p=5$ can be handled similarly.
In case (1), there is at most one $k$-isomorphism class of curves, with at most four twists over $\FF_{p^r}$. 

In cases (2)-(6), the curves are superspecial by \cite[Proposition 1.3]{IKO}. 
The singularities of $\mathcal{A}_{2,ss}$ are ordinary $(p+1)$-points which occur precisely at the superspecial points \cite[page 193]{kob}. 
There are $\ll p^2$ irreducible components of $\mathcal{A}_{2,ss}$,
each containing $p^2+1$ superspecial points by \cite[page 154]{KO}.
So the number of superspecial points in 
$\mathcal{A}_{2,ss}(k)$ is $\ll p^2(p^2+1)/(p+1) \ll p^3$.  
(See \cite[Theorem 2]{IbKat} for an exact formula in terms of class numbers.)

Applying \cite[Lemma 9.1]{achterhowe}, the number of $\FF_q$-models for 
superspecial curves of genus $2$ is also $\ll p^3$.  
This completes the proof since ${\rm lim}_{r \to \infty} p^3/p^r = 0$.
\end{proof}

\begin{remark}
The conclusion of Corollary \ref{Pmaincor} is false when $p=2$ by \cite[Theorem 3.1]{VdGVdV92}. 
\end{remark}


\section{Analysis in low  dimension: genus $3$ curves for $p=2$} \label{SVR}

Let $p=2$ and $k = \overline{\mathbb{F}}_2$. For $c,d \in k^*$, 
consider the generalized Artin-Schreier curve $X_{c,d}$ with affine equation
\begin{equation}\label{Eartsch}
X_{c, d}: Z^4 + (1 + c) Z^2 + c Z = d S^3.
\end{equation}

The cover $\gamma:X_{c, d} \to {\mathbb P}^1$, taking $(Z, S) \mapsto S$ is ramified only 
above $S=\infty$, where it is totally ramified.  The filtration of higher ramification groups 
trivializes at index $3$.  So by the Riemann-Hurwitz formula, $X_{c, d}$ has genus $3$.
By Lemma \ref{otherquotients}, $X_{c, d}$ is supersingular.
Let $q = 2^r$ be such that $c,d \in K=\FF_q$.

In the main result of the section, we determine the type of $X_{c, d}$ over $K$.
To state this, we set some notation.
Let $K'=\FF_q(h)$, where $h \in \FF_{q^2}$ is such that $h^2+h=c$. 
Then $h \in \FF_q$ if and only if ${\rm Tr}_r(c)=0$, 
where ${\rm Tr}_{r}:\FF_{2^r} \to \FF_2$ is the trace map.  Let $q'=2^{r'}=|K'|$.

\begin{theorem}\label{mainthm}
Let $X_{c, d}$, $r$ and $h$ be as defined above.
\begin{enumerate}
\item If $r$ is odd, then $X_{c, d}/K$ is fully maximal if $h \in \FF_{q}$ and mixed if $h \not\in \FF_{q}$. 
\item If $r \equiv 2 \bmod 4$, then $X_{c, d}/K$ is mixed if $h \in \FF_{q}$ and fully minimal if $h \not\in \FF_{q}$.
\item If $r \equiv 0 \bmod 4$, then $X_{c, d}/K$ is fully minimal.
\end{enumerate}
Moreover, $\mathrm{Jac}(X_{c, d})$ has the same type as $X_{c, d}$ over $K$, 
unless $r \equiv 0 \bmod 4$ and $h \in \FF_{q}$, in which case $\mathrm{Jac}(X_{c, d})$ is mixed.
\end{theorem}

\begin{remark}
\begin{enumerate}
\item If $d=d_1 d_2^3$ with $d_1,d_2 \in K$, there is an $\FF_q$-isomorphism $X_{c,d} \stackrel{\simeq}{\to} X_{c,d_1}$, 
taking $(Z, S) \mapsto (Z, S/d_2)$.
So $d$ can be replaced by any representative of the coset $d(K^*)^3$ in $K^*$;
if $r$ is odd, then one can set $d=1$.

\item The supersingular locus $S_3$ of the moduli space ${\mathcal M}_3 \otimes \FF_2$ has dimension $2$.  By part (1), the curves in the family $X_{c,d}$ are represented by a $1$-dimensional subspace of $S_3$. 
This $1$-dimensional family is the same as the one given in \cite[pages 56-57]{vianarodriguez} by 
\[X'_{a,b}: x+y + a(x^3y + xy^3) + bx^2y^2=0,\]
via the change of coordinates:
$c=a/b, \ d = a^3/b, \ S = 1/a(x+y), \ Z = x/(x+y)$.

\item
The proportion of $c \in \FF_q^*$ for which 
$X_{c, d}$ is mixed is a bit larger than $\frac{1}{2}$ when $r$ is odd and
a bit smaller than $\frac{1}{2}$ when $r \equiv 2 \bmod 4$
since $\#\{c \in \FF_q^* \mid {\rm Tr}_r(c) = 1\} =\frac{q}{2}$.

\end{enumerate}
\end{remark}

\subsection{Decomposition of the Jacobian}\mbox{}\\

Define the values
\begin{equation}\label{cidef}
c_1=d/c^2, \ c_2 = d/(h+1)^2, \ {\rm and} \ c_3=d/h^2,
\end{equation}
and corresponding elliptic curves
\begin{equation}\label{E123eq}
E_1: R^2 + R = c_1S^3, \ E_2:T^2 + T = c_2 S^3, \ E_3: U^2 + U = c_3 S^3.
\end{equation}
Also, define commuting order $2$ automorphisms on $X_{c, d}$ by:
\begin{equation}\label{E123aut}
\tau:(S, Z) \mapsto (S, Z+1) \ {\rm and} \ \upsilon:(S, Z) \mapsto (S, Z+h).
\end{equation}
Note that $\tau$ is defined over $K=\FF_q$ and $\upsilon$ is defined over $K'$.

\begin{lemma}\label{otherquotients}
\begin{enumerate}
\item Over $K$, the quotient of $X_{c,d}$ by $\tau$ is $E_1$.\\
Over $K'$, the quotient of $X_{c, d}$ by $\upsilon$ (resp.\ $\tau \upsilon$) is $E_2$ (resp.\ $E_3$).
\item Hence, ${\rm Jac}(X_{c, d}) \sim_{K'} E_1 \oplus E_2 \oplus E_3$ and $X_{c, d}$ is supersingular.
\item Thus $L(X_{c, d}/K',T) = L(E_1/K',T)L(E_2/K',T)L(E_3/K',T)$. 
\end{enumerate}
\end{lemma}

\begin{proof}
\begin{enumerate}
\item The involution $\tau$ fixes the function $R_1 = Z(Z+1)$. 
Similarly, the involutions $\upsilon$ and $\tau \upsilon$ fix the functions $T_1=Z(Z+h)$ and $U_1=Z(Z+(h+1))$ respectively.
Direct calculations show that: 
\begin{eqnarray*}
R_1^2+cR_1 & = & Z^4+(1+c)Z^2 + cZ = d S^3;\\
T_1^2+(h+1)T_1 & = & Z^4+h^2Z^2 + (h+1)(Z^2 + hZ)= d S^3;\\
U_1^2+hU_1 &= &Z^4+(h+1)^2 Z^2 + h(Z^2+(h+1)Z) = d S^3.
\end{eqnarray*}
Setting $R_1=cR$, $T_1=(h+1) T$, and $U_1=hU$, then
\begin{equation*}
R^2+R = (d/c^2) S^3, \ T^2+T=(d/(h+1)^2)S^3, \ U^2+U=(d/h^2)S^3.
\end{equation*}

\item The decomposition is immediate from part (1) and \cite[Theorem~B]{kanirosen}.
By the Deuring-Shafarevich formula, $E_1,E_2,E_3$ have $2$-rank $0$ and hence are supersingular.
Thus $X_{c, d}$ is supersingular by Theorem \ref{propssAV}. 

\item This is immediate from part (2). 
\end{enumerate}
\end{proof}


\subsection{The \NWN of $E_1$, $E_2$, and $E_3$} \label{SSLpolyEab} \mbox{}\\

\begin{lemma} \label{LbasecaseE}
The elliptic curve $E_\circ: R^2+R=S^3$ is maximal over $\FF_{2^2}$ and
\[L(E_\circ/\FF_2, T)=1+2T^2=(1-(\sqrt{2} i) T)(1-(-\sqrt{2} i) T).\] 
\end{lemma}

\begin{lemma} \label{LquotientELpolycube}
\begin{enumerate}
\item If $c_1$ is a cube in $K^*$, then $\NWNs(E_1/K) = \{i^r, (-i)^r\}$.

\item For $j=2,3$, if $c_j$ is a cube in $(K')^*$, then 
$\NWNs(E_j/K') = \{i^{r'}, (-i)^{r'}\}$.
\end{enumerate}
\end{lemma}

\begin{proof}
If $c_1$ is a cube in $K^*$, then there is an isomorphism $w:E_1 \to E_\circ$ defined over $K$,
so part (1) follows from Lemmas \ref{baseAV} and \ref{LbasecaseE}.  
The proof for part (2) is similar.
\end{proof}

\begin{lemma} \label{LquotientELpolynotcube}
\begin{enumerate}
\item Suppose that $c_1$ is not a cube in $K^*$.
If $r \equiv 2 \bmod 4$, then $\NWNs(E_1/K)$ is $\{\zeta_6, \bar{\zeta}_6\}$ or 
$\{-1,-1\}$.
If $r \equiv 0 \bmod 4$, then $\NWNs(E_1/K)$ is $\{\zeta_3, \bar{\zeta}_3\}$ or $\{1,1\}$.
\item Suppose that $c_j$ is not a cube in $(K')^*$ for $j=2,3$.
If $r' \equiv 2 \bmod 4$, then $\NWNs(E_j/K')$ is $\{\zeta_6, \bar{\zeta}_6\}$ or $\{-1,-1\}$.
If $r' \equiv 0 \bmod 4$, then $\NWNs(E_j/K')$ is $\{\zeta_3, \bar{\zeta}_3\}$ or $\{1,1\}$.
\end{enumerate}
\end{lemma}

\begin{proof}
For part (1), if $c_1$ is not a cube in $K^*$, then it is a cube in $K_3^*$, where 
$K_3 \simeq \FF_{q^3}$.
By Lemma \ref{LquotientELpolycube}(1), $\NWNs(E_1/K_3)=\{i^{3r}, (-i)^{3r}\}=\{i^r, (-i)^{r}\}$.  
If $r \equiv 2 \bmod 4$, then $\NWNs(E_1/K_3)=\{-1,-1\}$, 
while if $r \equiv 0 \bmod 4$, then $\NWNs(E_1/K_3)=\{1,1\}$.
By Lemma \ref{baseAV}, $\NWNs(E_1/K)$ are the cube roots of $\NWNs(E_1/K_3)$ and are complex conjugates.
The proof for part (2) is similar.
\end{proof}

Lemmas \ref{otherquotients}(3), \ref{LquotientELpolycube}, and \ref{LquotientELpolynotcube} determine $\underline{e}(X_{c,d}/K')$.
When $h \not \in \FF_q$, this is not quite strong enough 
to prove Theorem \ref{mainthm}, because it only gives information about the
\NWN over $\FF_{q^2}$.  
We now determine more information using the Artin $L$-series $L(E_1/\FF_{q},T,\chi)$, 
where $\chi$ is the nontrivial character of $\ZZ/2\ZZ$.
By Lemma \ref{eulergeneral}(1) (\cite[Chapter 9, page 130]{rosen}),
\begin{equation}\label{rosen} 
L(X_{c,d}/\FF_q, T)=L(E_1/\FF_q, T) L(E_1/\FF_q, T, \chi).
\end{equation}

Let $\rho_1$ be the coefficient of $T$ in $L(E_1/K, T, \chi)$. 
Let $I_1$ (resp.\ $S_1$) be the number of $K$-points of $E_1$ 
that are inert (resp.\ split) in $X_{c,d}$. 
By Lemma \ref{eulergeneral}(2), $\rho_1=S_1-I_1$.
The conditions of Remark \ref{R22twist}(3) are satisfied if ${\rm Tr}_{r}(c)=1$, 
so $\rho_1=0$ by \eqref{Eambiguity}.

%

%

\begin{proposition} \label{Pg3p2vecE}
Let $K = \FF_{q}$ where $q=2^r$.
Let $K'=K(h)$ where $h$ is such that $h^2+h = c$.
The $2$-valuation vector $\underline{e}(X_{c,d}/K) =\{e_1,e_2,e_3\}$ is determined below.
\begin{center}
    \begin{tabular}{| l | l | l | l |}
    \hline
 $\underline{e}$ & $r$ {\rm odd} & $r \equiv 2 \bmod 4$ & $r \equiv 0 \bmod 4$  \\ \hline
if $h \in \FF_{q}$ & $\{2,2,2\}$  & $\{1,1,1\}$ & $\{0,0,0\}$ \\ \hline
if $h \not \in \FF_q$ & $\{2,2,2\}$ & $\{1,0,1\}$ & $\{0,0,1\}$ \\ \hline
\end{tabular}
\end{center}
\end{proposition}

\begin{proof}
When $h \in \FF_{q}$, then $K'=K$.
By Lemmas \ref{LquotientELpolycube} and \ref{LquotientELpolynotcube}, 
$\NWNs(X_{c,d}/K)$ are among the values $(\pm i)^r$ if $r$ is odd, 
$\zeta_6, \bar{\zeta}_6, -1$ if $r \equiv 2 \bmod 4$, 
and $\zeta_3, \bar{\zeta}_3, 1$ if $r \equiv 0 \bmod 4$. 
Thus $\underline{e}(X_{c,d}/K)$ equals $\{2\}$ if $r$ is odd, 
$\{1\}$ if $r \equiv 2 \bmod 4$, and $\{0\}$ if $r \equiv 0 \bmod 4$.

Suppose that $h \not\in \FF_{q}$. Then $\NWNs(E_1/K)$ are the same as 
before; in particular, $e_1=2$ if $r$ is odd, $e_1=1$ if $r \equiv 2 \bmod 4$,
and $e_1=0$ if $r \equiv 0 \bmod 4$.
By Lemmas \ref{LquotientELpolycube} and \ref{LquotientELpolynotcube}, 
$\NWNs(E_2/K')$ and $\NWNs(E_3/K')$ are among
$-1$ and $\zeta_6^{\pm 1}$ if $r$ is odd,
and $1$ and $\zeta_3^{\pm 1}$ if $r$ is even.  
Since $K'$ is a quadratic extension of $K$,
$\NWNs(E_2/K)$ and $\NWNs(E_3/K)$ are among the square roots of these.
The ambiguity in taking the square root is resolved by the fact that 
the four sum to zero by \eqref{Eambiguity} and are invariant under complex conjugation.
If $r$ is odd, then $\NWNs(E_2/K) \cup \NWNs(E_3/K)$ is either $\{\pm i, \pm i\}$
or $\{\zeta_{12},\zeta_{12}^5,\zeta_{12}^7,\zeta_{12}^{11}\}$, which both yield 
$\{e_2,e_3\}=\{2,2\}$.
If $r$ is even, then $\NWNs(E_2/K)\cup \NWNs(E_3/K)$ is either $\{1,1,-1,-1\}$ or 
$\{\zeta_6, \zeta_6^{-1}, \zeta_3, \zeta_3^{-1}\}$ which both yield 
$\{e_2,e_3\}=\{0,1\}$.
\end{proof}

\subsection{The automorphism group of $X_{c,d}$ and $K$-Frobenius conjugacy classes}\mbox{}\\

Let $G={\rm Aut}_k(X_{c,d})$.  Recall $\tau$ and $\upsilon$ from \eqref{E123aut}.
Let $S_0 = \langle \tau, \upsilon \rangle \simeq \ZZ/2\ZZ \times \ZZ/2\ZZ$. 

Consider the order $3$ automorphism of $X_{c,d}$, given by $\sigma:(S,Z) \mapsto (\zeta_3^2 S, Z)$.
Note that $\sigma$ is defined over $\FF_q$ if $r$ is even and over $\FF_{q^2}$ if $r$ is odd.
Furthermore, $\sigma$ centralizes $S_0$.

\begin{lemma}\label{genericaut}
If $c \not = 1$, then $G=S_0 \times \langle \sigma \rangle$ is an abelian group of order $12$.
If $c=1$, then $G$ is a semidirect product of the form $S_0 \rtimes H$ 
where $H$ is a cyclic group of order $9$. 
\end{lemma}

\begin{proof}
The degree $4$ equation \eqref{Eartsch} for $X_{c,d}$ is of the type whose automorphism group is 
studied in  \cite{stich}, see also \cite[Section 12.1]{Kbook}.
By \cite[Theorem~12.11]{Kbook}, $G$ fixes the unique point of $X_{c,d}$ lying above $S=\infty$.
Thus $G \simeq S_1 \rtimes H$ where $S_1$ is the normal Sylow $2$-subgroup of $G$ and
$H$ is a cyclic group of odd order.
By \cite[Theorem~12.7]{Kbook}, $\vert S_1\vert =4$ (so $S_1 = S_0$)
and $\vert H\vert $ divides $9$.  Then $\vert H\vert  = 3$ or $9$ since $\sigma \in G$. 

If $H$ contains an element $\kappa$ of order $9$, then $\kappa(S) = \zeta_9 S$.
Hence, $\kappa$ acts on the right hand side of \eqref{Eartsch} by multiplication by $\zeta_3$.
However, $\kappa$ can only act on the left hand side of \eqref{Eartsch} by multiplication by $\zeta_3$ 
if the monomial $(1+c) x^2$ vanishes.  
Thus, $\kappa$ lifts to an automorphism of $X_{c,d}$ if and only if $c=1$, in which case
$\kappa(Z)=\zeta_3 Z$ and $\kappa : (S,Z) \mapsto (\zeta_9 S, \zeta_3 Z)$.
\end{proof}

If $c=1$ and $\vert H\vert =9$, note that $\kappa^3=\sigma^2$; also
$G$ is non-abelian, since $\kappa \tau \kappa^{-1}(Z)=Z+\zeta_3^{-1}$, so $\kappa \tau \kappa^{-1}$ 
is either $\upsilon$ or $\upsilon \tau$, depending on the choice of 
$h \in \{\zeta_3, \zeta_3^2\}$.
In this case, $\kappa$ permutes the three quotients $E_1, E_2, E_3$ of $X_{c,d}$ by the non-trivial involutions in $S_0$.

Let $Fr=Fr_{K}$ where $K=\FF_q$.
We now determine the $K$-Frobenius conjugacy classes of~$G$. 

\begin{lemma}\label{frobclasses} 
Let $f$ be the number of $K$-Frobenius conjugacy classes in $G$.
\begin{enumerate}
\item Suppose that $c \not = 1$.  Then $G$ is an abelian group of order $12$.

\begin{enumerate}
\item If $r$ is even and $h \in \FF_q$, then $f=12$.
\item If $r$ is even and $h \not \in \FF_q$, then $f=6$.\\
The classes are $\{\mathrm{id},\tau\}$,$\{\upsilon,\upsilon\tau\}$,$\{\sigma,\sigma\tau\}$, $\{\upsilon\sigma, \upsilon\tau\sigma\}$, $\{\sigma^2,\sigma^2\tau\}$,$ \{\upsilon\sigma^2,\upsilon\tau\sigma^2\}$.
\item If $r$ is odd and $h \in \FF_q$, then $f=4$.\\
The classes are
are $\{\mathrm{id},\sigma,\sigma^2\}$, $\{\upsilon,\upsilon\sigma,\upsilon\sigma^2\}$, $\{\tau,\tau\sigma, \tau\sigma^2\}$, and $\{\upsilon\tau, \upsilon\tau\sigma, \upsilon\tau\sigma^2\}$.
\item If $r$ is odd and $h \not \in \FF_q$, then $f=2$.\\
The classes are $\{\mathrm{id},\sigma,\sigma^2,\tau,\tau\sigma,\tau\sigma^2\}$ and $\{\upsilon,\upsilon\sigma,\upsilon\sigma^2,\upsilon\tau,\upsilon\tau\sigma, \upsilon\tau\sigma^2\}$.
\end{enumerate}
\item If $c=1$, then $G$ is a non-abelian group of order $36$ and $h \in \FF_4-\FF_2$. 
\begin{enumerate}
\item If $r$ is even, then $h \in \FF_q$ and $f = 10$.\\
The classes are $\{ \mathrm{id} \}$, $\{\upsilon,\tau,\upsilon\tau \}$, and $\{\kappa^j, \upsilon\kappa^j, \tau\kappa^j, \upsilon\tau\kappa^j\}$ for $j = 1,\ldots,8$.
\item If $r$ is odd, then $h \not \in \FF_q$ and $f = 2$.  
Also, $\upsilon$ is not conjugate to ${\rm id}$.\\
The first class is
$\{\mathrm{id}, \tau, \kappa^1, \ldots, \kappa^8, \upsilon\tau\kappa, \upsilon\kappa^2, \tau\kappa^3, \upsilon\tau\kappa^4,\upsilon\kappa^5, \tau\kappa^6, \upsilon\tau\kappa^7, \upsilon\kappa^8\}$.
\end{enumerate}
\end{enumerate}
\end{lemma}

\begin{proof}
We omit most of the long calculation. 
Cases (1a) and (2a) follow from the fact that
$K$-Frobenius conjugacy classes coincide with standard conjugacy classes
when all automorphisms are defined over $K$.  

For the other cases, note that $^{Fr} \tau = \tau$.
If $h \in \FF_q$, then $^{Fr} \upsilon = \upsilon$.
If $h \not \in \FF_q$, then $h^q=h+1$ and $^{Fr}\upsilon = \upsilon \tau$;
in this case, $ \upsilon^{-1}\tau(^{Fr}\upsilon) = {\rm id}$, showing that $\tau$ is $K$-Frobenius conjugate to ${\rm id}$, 
and $\upsilon$ is $K$-Frobenius conjugate to $\upsilon\tau$.

Also, $^{Fr}\kappa = \kappa^{q}$.
If $r$ is even, then $^{Fr} \sigma = \sigma$. 
If $r$ is odd, then $^{Fr} \sigma = \sigma^{-1}$; 
in this case, $ \sigma^{-1}{\rm id}(^{Fr} \sigma) = \sigma$, showing that $\sigma$ is $K$-Frobenius conjugate
to ${\rm id}$.
\end{proof}

\newpage
\subsection{Proof of Theorem \ref{mainthm}} \label{SVRfinal}

\begin{proof}[Proof of Theorem \ref{mainthm}]
The results from Remark \ref{R22twist} apply here, by setting $S = S_0$. 
By Lemma \ref{otherquotients}(2), ${\rm Jac}(X_{c,d}) \sim_{K'} E_1 \oplus E_2 \oplus E_3$. 
By Lemma \ref{otherquotients}(1) and Remark \ref{R22twist}(1), over $K'$, the automorphism $\tau$ acts trivially on $E_1$ and by $[-1]$ on $E_2$ and $E_3$; similarly, $\upsilon$ fixes $E_2$ and acts by $[-1]$ on $E_1$ and $E_3$, 
and $\upsilon\tau$ fixes $E_3$ and acts by $[-1]$ on $E_1$ and $E_2$.

When $h \not\in \FF_q$, the strategy in the proof below
is to analyze the situation for the base change to $K'$, where the automorphism 
$g$ acts via $g^{Fr_K}g$.  
The ambiguity caused by descending to $K$ can be resolved using \eqref{Eambiguity}.

In each case below, the information on $\NWNs(X_{c,d}/K)$ for $K = \FF_{q}$ and their $2$-adic valuations $\underline{e}= \underline{e}(X_{c,d}/K)=\{e_1,e_2,e_3\}$ is from Proposition \ref{Pg3p2vecE}. 
The data on the number and representatives of the $K$-twists of $X_{c,d}$ are found in Lemma \ref{frobclasses}.

\begin{enumerate}
\item Let $r$ be odd.  Then $\underline{e} = \{2,2,2\}$ so $X_{c,d}$ has parity $+1$. 
\begin{enumerate}
\item 
If $h \in \FF_{q}$, then there are three nontrivial twists, each of order $2$.
By Lemma \ref{Levenparitychange}, none of these change the parity, 
so $X_{c,d}$ is fully maximal.
\item If $h \not\in \FF_{q}$, then $K'=\FF_{q^2}$. The nontrivial $K$-twist is represented by $\upsilon$ (which is not defined over $\FF_q$). 
Then $\underline{e}(X_{c,d}/K')=\{1,1,1\}$. 
Over $K'$, the twist for $\upsilon$ corresponds to $\upsilon^{Fr_K}\upsilon = \tau$, 
which negates the two conjugate pairs of \NWN for $E_2$ and $E_3$,
thus the twist has $\underline{e}(X'_{c,d}/K')=\{1,0,0\}$.
By \eqref{Eambiguity}, 
$\underline{e}(X'_{c,d}/K)=\{2,0,1\}$, of parity $-1$.
Thus, $X_{c,d}$ is mixed.
\end{enumerate}
In addition, $\mathrm{Jac}(X_{c,d})$ and $X_{c,d}$ have the same type, by Lemma \ref{xjac}.
\item Let $r \equiv 2 \bmod 4$.
\begin{enumerate}
\item If $h \in \FF_{q}$, then $\underline{e} = \{1,1,1\}$, so $X_{c,d}$ has parity $+1$. There are either twelve $K$-twists (if $c \neq 1$) or ten $K$-twists (if $c=1$). 
In both cases, the $K$-twist by $\upsilon$ has $\underline{e} = \{0,1,0\}$ and parity $-1$. Hence, both $X_{c,d}$ and $\mathrm{Jac}(X_{c,d})$ are mixed.
\item If $h \not\in \FF_{q}$, then $\underline{e} = \{1,0,1\}$, so $X_{c,d}$ has parity $-1$. Also, $\underline{e}(X_{c,d}/K') = \{0,0,0\}$.
Since $c \neq 1$, there are six $K$-twists, represented by $\mathrm{id}$, $\upsilon$, $\sigma$, $\upsilon\sigma$, $\sigma^2$, and $\upsilon\sigma^2$. 
Twisting by $\mathrm{id}, \sigma, \sigma^2$ does not change the parity by Lemma \ref{Loddnotparity} 
since these automorphisms have odd order and are defined over $K$. 
The twist of $X_{c,d}/K$ by $\upsilon$ (resp.\ $\upsilon \sigma$, $\upsilon \sigma^2$)
corresponds to the twist of $X_{c,d}/K'$ by $\tau$ (resp. $\tau \sigma^2$, $\tau \sigma$), 
which changes $\underline{e}(X_{c,d}/K')$ to $\{0,1,1\}$.  
So the $K$-twist for $\upsilon$ (resp.\ $\upsilon \sigma$, $\upsilon \sigma^2$) has $\underline{e}(X_{c,d}/K)$
either $\{1,2,2\}$ or $\{0,2,2\}$, which both have parity $-1$.
Thus $X_{c,d}$ is fully minimal over $K$.
The twist by $[-1]$ has $\underline{e}=\{0,1,0\}$, thus $\mathrm{Jac}(X_{c,d})$ is fully minimal as well.
\end{enumerate} 

\item Let $r \equiv 0 \bmod 4$.
\begin{enumerate}
\item If $h\in \FF_{q}$, then $\underline{e} = \{0,0,0\}$, so $X_{c,d}$ has parity $-1$. 
There are either twelve $K$-twists (if $c \not = 1$) or ten $K$-twists (if $c = 1$). 
The nontrivial elements of $S_0$ yield twists such that $\underline{e}=\{1,1,0\}$, of parity $-1$, cf.\ Remark \ref{R22twist}(1).
The odd order automorphisms $\sigma^j$ do not change the parity by Lemma \ref{Loddnotparity}.
If $c \not = 1$, then all automorphisms are defined over $K$ and the group is abelian, so no other twist changes the parity either.
If $c=1$, then the twists by $\kappa^j$ permute $E_1$, $E_2$, $E_3$ and thus do not change the parity either.
So $X_{c,d}$ is fully minimal.  Since $\mathrm{Jac}(X_{c,d})$ has a twist with $\underline{e}=\{1,1,1\}$ and parity $+1$, it is mixed.

\item  If $h \not\in \FF_{q}$, then $\underline{e} = \{0,0,1\}$, so $X_{c,d}$ has parity $-1$. 
The proof that both $X_{c,d}$ and $\mathrm{Jac}(X_{c,d})$ are fully minimal is very similar to case (2b).
\end{enumerate}
\end{enumerate}
\end{proof}


\begin{thebibliography}{10}

\bibitem{achterhowe}
Jeffrey~D. Achter and Everett~W. Howe, \emph{Split abelian surfaces over finite
  fields and reductions of genus-2 curves}, Algebra Number Theory \textbf{11}
  (2017), no.~1, pp.\ 39--76.

\bibitem{Bwin}
Irene Bouw, Wei Ho, Beth Malmskog, Renate Scheidler, Padmavathi Srinivasan, and
  Christelle Vincent, \emph{Zeta functions of a class of {A}rtin-{S}chreier
  curves with many automorphisms}, Directions in number theory, Assoc. Women
  Math. Ser., vol.~3, Springer, [Cham], 2016, pp.~87--124. \MR{3596578}

\bibitem{cardona}
Gabriel Cardona, \emph{On the number of curves of genus 2 over a finite field},
  Finite Fields Appl. \textbf{9} (2003), no.~4, pp.\ 505--526.

\bibitem{CN}
Gabriel Cardona and Enric Nart, \emph{Zeta function and cryptographic exponent
  of supersingular curves of genus 2}, Pairing-based cryptography---{P}airing
  2007, Lecture Notes in Comput. Sci., vol. 4575, Springer, Berlin, 2007,
  pp.~132--151.

\bibitem{CHdescent}
Jean-Marc Couveignes and Emmanuel Hallouin, \emph{Global descent obstructions
  for varieties}, Algebra Number Theory \textbf{5} (2011), no.~4, pp.\
  431--463.

\bibitem{deligne}
Pierre Deligne, \emph{La conjecture de {W}eil. {I}}, Inst. Hautes \'Etudes Sci.
  Publ. Math. (1974), no.~43, pp.\ 273--307.

\bibitem{glasspries}
Darren Glass and Rachel Pries, \emph{Hyperelliptic curves with prescribed
  {$p$}-torsion}, Manuscripta Math. \textbf{117} (2005), no.~3, pp.\ 299--317.

\bibitem{GGR}
Josep Gonz\'alez, Jordi Gu\`ardia, and Victor Rotger, \emph{Abelian surfaces of
  {${\rm GL}_2$}-type as {J}acobians of curves}, Acta Arith. \textbf{116}
  (2005), no.~3, pp.\ 263--287.

\bibitem{HashI}
Ki-ichiro Hashimoto and Tomoyoshi Ibukiyama, \emph{On class numbers of positive
  definite binary quaternion {H}ermitian forms. {II}}, J. Fac. Sci. Univ. Tokyo
  Sect. IA Math. \textbf{28} (1981), no.~3, pp.\ 695--699.

\bibitem{Kbook}
J.~W.~P. Hirschfeld, G.~Korchm\'aros, and F.~Torres, \emph{Algebraic curves
  over a finite field}, Princeton Series in Applied Mathematics, Princeton
  University Press, Princeton, NJ, 2008. \MR{2386879}

\bibitem{honda}
Taira Honda, \emph{Isogeny classes of abelian varieties over finite fields}, J.
  Math. Soc. Japan \textbf{20} (1968), pp.\ 83--95.

\bibitem{HLP}
Everett~W. Howe, Franck Lepr\'evost, and Bjorn Poonen, \emph{Large torsion
  subgroups of split {J}acobians of curves of genus two or three}, Forum Math.
  \textbf{12} (2000), no.~3, pp.\ 315--364.

\bibitem{HMNR}
Everett~W. Howe, Daniel Maisner, Enric Nart, and Christophe Ritzenthaler,
  \emph{Principally polarizable isogeny classes of abelian surfaces over finite
  fields}, Math. Res. Lett. \textbf{15} (2008), no.~1, pp.\ 121--127.

\bibitem{ibu3}
Tomoyoshi Ibukiyama, \emph{On rational points of curves of genus {$3$} over
  finite fields}, Tohoku Math. J. (2) \textbf{45} (1993), no.~3, pp.\ 311--329.

\bibitem{IbKat}
Tomoyoshi Ibukiyama and Toshiyuki Katsura, \emph{On the field of definition of
  superspecial polarized abelian varieties and type numbers}, Compositio Math.
  \textbf{91} (1994), no.~1, pp.\ 37--46.

\bibitem{IKO}
Tomoyoshi Ibukiyama, Toshiyuki Katsura, and Frans Oort, \emph{Supersingular
  curves of genus two and class numbers}, Compositio Math. \textbf{57} (1986),
  no.~2, pp.\ 127--152.

\bibitem{Igusa}
Jun-ichi Igusa, \emph{Class number of a definite quaternion with prime
  discriminant}, Proc. Nat. Acad. Sci. U.S.A. \textbf{44} (1958), pp.\
  312--314.

\bibitem{kanirosen}
Ernst Kani and Michael Rosen, \emph{Idempotent relations and factors of {J}acobians},
  Math. Ann. \textbf{284} (1989), no.~2, pp.\ 307--327.

\bibitem{KO}
Toshiyuki Katsura and Frans Oort, \emph{Families of supersingular abelian
  surfaces}, Compositio Math. \textbf{62} (1987), no.~2, pp.\ 107--167.

\bibitem{kob}
Neal Koblitz, \emph{{$p$}--adic variation of the zeta-function over families of
  varieties defined over finite fields}, Compositio Math. \textbf{31} (1975),
  no.~2, pp.\ 119--218.

\bibitem{Lang}
Serge Lang, \emph{Abelian varieties}, Interscience Tracts in Pure and Applied
  Mathematics. No. 7, Interscience Publishers, Inc., New York; Interscience
  Publishers Ltd., London, 1959.

\bibitem{lauter}
Kristin Lauter, \emph{Geometric methods for improving the upper bounds on the
  number of rational points on algebraic curves over finite fields}, J.
  Algebraic Geom. \textbf{10} (2001), no.~1, pp.\ 19--36, with an appendix in
  French by J.-P. Serre.

\bibitem{lioort}
Ke-Zheng Li and Frans Oort, \emph{Moduli of supersingular abelian varieties},
  Lecture Notes in Mathematics, vol. 1680, Springer-Verlag, Berlin, 1998.

\bibitem{maisnart}
Daniel Maisner and Enric Nart, \emph{Abelian surfaces over finite fields as
  {J}acobians}, Experiment. Math. \textbf{11} (2002), no.~3, pp.\ 321--337,
  with an appendix by Everett W. Howe.

\bibitem{manin2}
Juri I. Manin, \emph{Theory of commutative formal groups over fields of finite
  characteristic}, Uspehi Mat. Nauk \textbf{18} (1963), no.~6 (114), pp.\
  3--90.

\bibitem{meatop}
Stephen Meagher and Jaap Top, \emph{Twists of genus three curves over finite
  fields}, Finite Fields Appl. \textbf{16} (2010), no.~5, pp.\ 347--368.

\bibitem{milneAV}
James~S. Milne, \emph{Abelian varieties (v2.00)}, 2008, available at
  www.jmilne.org/math/.

\bibitem{milneJac}
\bysame, \emph{Jacobian varieties}, 2012, available at
  http://www.jmilne.org/math/.

\bibitem{oortsub}
Frans Oort, \emph{Subvarieties of moduli spaces}, Invent. Math. \textbf{24}
  (1974), pp.\ 95--119.

\bibitem{oortNP}
\bysame, \emph{Abelian varieties over finite fields}, Higher-dimensional
  geometry over finite fields, NATO Sci. Peace Secur. Ser. D Inf. Commun.
  Secur., vol.~16, IOS, Amsterdam, 2008, pp.~123--188.

\bibitem{rosen}
Michael Rosen, \emph{Number theory in function fields}, Graduate Texts in
  Mathematics, vol. 210, Springer-Verlag, New York, 2002.

\bibitem{schoof}
Ren{\'e} Schoof, \emph{Nonsingular plane cubic curves over finite fields}, J.
  Combin. Theory Ser. A \textbf{46} (1987), no.~2, pp.\ 183--211.

\bibitem{serreloc}
Jean-Pierre Serre, \emph{Local fields}, Graduate Texts in Mathematics, vol.~67,
  Springer-Verlag, New York-Berlin, 1979, translated from the French by Marvin
  Jay Greenberg.

\bibitem{serregal}
\bysame, \emph{Galois cohomology}, Springer-Verlag, Berlin, 1997, translated
  from the French by Patrick Ion and revised by the author.

\bibitem{aec}
Joseph~H. Silverman, \emph{The arithmetic of elliptic curves}, second ed.,
  Graduate Texts in Mathematics, vol. 106, Springer, Dordrecht, 2009.

\bibitem{stich}
Henning Stichtenoth, \emph{\"{U}ber die {A}utomorphismengruppe eines
  algebraischen {F}unktionenk\"orpers von {P}rimzahlcharakteristik. {II}. {E}in
  spezieller {T}yp von {F}unktionenk\"orpern}, Arch. Math. (Basel) \textbf{24}
  (1973), pp.\ 615--631.

\bibitem{stichtenothII}
\bysame, \emph{Algebraic function fields and codes}, Universitext,
  Springer-Verlag, Berlin, 1993.

\bibitem{stichxing}
Henning Stichtenoth and Chao~Ping Xing, \emph{On the structure of the divisor
  class group of a class of curves over finite fields}, Arch. Math. (Basel)
  \textbf{65} (1995), no.~2, pp.\ 141--150.

\bibitem{tate}
John Tate, \emph{Endomorphisms of abelian varieties over finite fields},
  Invent. Math. \textbf{2} (1966), pp.\ 134--144.

\bibitem{tate71}
\bysame, \emph{Classes d'isog\'enie des vari\'et\'es ab\'eliennes sur un corps
  fini (d'apr\`es {T}. {H}onda)}, S\'eminaire {B}ourbaki. {V}ol. 1968/69:
  {E}xpos\'es 347--363, Lecture Notes in Math., vol. 175, Springer, Berlin,
  1971, pp.~95--110.

\bibitem{VdGVdV92}
Gerard van~der Geer and Marcel van~der Vlugt, \emph{Supersingular curves of
  genus {$2$} over finite fields of characteristic {$2$}}, Math. Nachr.
  \textbf{159} (1992), pp.\ 73--81.

\bibitem{VdGVdV}
\bysame, \emph{On the existence of supersingular curves of given genus}, J.
  Reine Angew. Math. \textbf{458} (1995), pp.\ 53--61.

\bibitem{vianarodriguez}
Paulo~H. Viana and Jaime E.~A. Rodriguez, \emph{Eventually minimal curves},
  Bull. Braz. Math. Soc. (N.S.) \textbf{36} (2005), no.~1, pp.\ 39--58.

\bibitem{water}
William~C. Waterhouse, \emph{Abelian varieties over finite fields}, Ann. Sci.
  \'Ecole Norm. Sup. (4) \textbf{2} (1969), pp.\ 521--560.

\bibitem{weil}
Andr{\'e} Weil, \emph{Sur les courbes alg\'ebriques et les vari\'et\'es qui
  s'en d\'eduisent}, Actualit\'es Sci. Ind., no. 1041, Hermann et Cie., Paris,
  1948.

\bibitem{weil2}
\bysame, \emph{Vari\'et\'es ab\'eliennes et courbes alg\'ebriques},
  Actualit\'es Sci. Ind., no. 1064, Hermann \& Cie., Paris, 1948.

\bibitem{XYY}
Jiangwei Xue, Tse-Chung Yang, and Chia-Fu Yu, \emph{On superspecial abelian
  surfaces over finite fields}, Doc. Math. \textbf{21} (2016), 1607--1643.
  \MR{3603930}

\end{thebibliography}

\def\cprime{$'$}
\providecommand{\bysame}{\leavevmode\hbox to3em{\hrulefill}\thinspace}
\providecommand{\MR}{\relax\ifhmode\unskip\space\fi MR }
\providecommand{\MRhref}[2]{%
  \href{http://www.ams.org/mathscinet-getitem?mr=#1}{#2}
}
\providecommand{\href}[2]{#2}

\end{document}